\documentclass[11pt]{article}
\usepackage{fullpage,amsthm,amssymb,amsmath}
\usepackage{graphicx,algorithmic,algorithm}%
\usepackage{enumerate}
\usepackage{tikz}
\usetikzlibrary{shapes}
\usepackage{hyperref}
\usepackage{subcaption}
\usepackage[normalem]{ulem}
\hypersetup{
    pdftitle=   {Linear rank-width of distance-hereditary graphs II. Vertex-minor obstructions},
   pdfauthor=  {Isolde Adler and Mamadou Kant\'e and O-joung Kwon}
}
\usepackage{authblk}
\usepackage{mathabx}
\newtheorem{theorem}{Theorem}[section]
\newtheorem{lemma}[theorem]{Lemma}
\newtheorem{proposition}[theorem]{Proposition}

\newtheorem*{thm:main}{Theorem~\ref{thm:main}}
\newenvironment{clproof}{\begin{list}{}{%
              \setlength{\leftmargin}{5mm}%
              } \item {\it Proof.} }{\hfill$\lozenge$\end{list}\medskip}

\newtheorem*{THMMAIN}{Theorem \ref{thm:largelrw}}
\newtheorem{CON}[theorem]{Conjecture}
\newtheorem{CLAIM}{Claim}
\newtheorem{QUE}{Question}

\theoremstyle{remark}

\theoremstyle{definition}

\newcommand\abs[1]{\lvert #1\rvert}

\newcommand\limb{\mathcal{L}}
\newcommand\limbhat{\mathcal{LG}}
\newcommand\limbtil{\mathcal{LC}}

\newcommand\dist{\operatorname{dist}}
\newcommand\rank{\operatorname{rank}}

\newcommand\lrw{\operatorname{lrw}}

\newcommand\pw{\operatorname{pw}}
\newcommand\cutrk{\operatorname{cutrk}}

\newcommand\obt{\Omega_T}

\newcommand\cB{\mathcal{B}}
\newcommand\cO{\mathcal{O}}

\newcommand{\macro}[3]{\newcommand{#1}[#3]{#2}}
\macro{\lbwd}{\operatorname{pw}(#1)}{1}
\macro{\bwd}{\operatorname{bw}(#1)}{1}

\newcommand{\bag}[2]{\textsf{bag}_{#1}(#2)}
\newcommand{\node}[2]{\textsf{node}_{#1}(#2)}

\newcommand{\origin}[1]{\mathcal{G}[#1]}

\def\cD{\mathcal{D}}

\begin{document}
\title{Linear rank-width of distance-hereditary graphs II. Vertex-minor obstructions}

\author[1]{Mamadou Moustapha Kant\'e}

\author[2]{O-joung Kwon\thanks{Supported by the European Research Council (ERC) under the European Union's Horizon 2020 research and innovation programme (ERC consolidator grant DISTRUCT, agreement No. 648527).}}

\affil[1]{Universit\'e Clermont Auvergne, LIMOS, CNRS, Aubi\`ere, France.}
\affil[2]{Logic and Semantics, Technische Universit\"at Berlin, Berlin, Germany.}
\date\today
\maketitle

\footnotetext{E-mail addresses: \texttt{mamadou.kante@uca.fr} (M. M. Kant\'e) and \texttt{ojoungkwon@gmail.com} (O. Kwon)}

\begin{abstract} 
	In the companion paper [Linear rank-width of distance-hereditary graphs I. A polynomial-time algorithm, Algorithmica 78(1):342--377, 2017], we presented a characterization of the linear rank-width of distance-hereditary graphs,
  	from which we derived an algorithm to compute it in polynomial time.  In this paper, we investigate structural properties of distance-hereditary graphs based on this characterization.

	 First, we prove that for a fixed tree $T$, 
  	every distance-hereditary graph of sufficiently large linear rank-width contains a vertex-minor isomorphic to $T$.  
 	We extend this property to bigger graph classes, namely, classes of graphs whose prime induced subgraphs have bounded linear rank-width.
 	Here, prime graphs are graphs containing no splits. 
 	We conjecture that for every tree $T$, every graph of sufficiently large linear rank-width contains a vertex-minor isomorphic to $T$.
	Our result implies that it is sufficient to prove this conjecture for prime graphs.

 	 For a class $\Phi$
  	of graphs closed under taking vertex-minors, a graph $G$ is called a \emph{vertex-minor obstruction} for $\Phi$ if $G\notin \Phi$ but all of its proper vertex-minors are contained in $\Phi$. 
  	Secondly, we provide, for each $k\ge 2$, a set of distance-hereditary graphs that contains all distance-hereditary vertex-minor
  	obstructions for graphs of linear rank-width at most $k$.  
    Also, we give a simpler way to obtain the known vertex-minor obstructions for graphs of linear rank-width at most $1$.
\end{abstract}

\section{Introduction} 

	\emph{Linear rank-width} is a linear-type width parameter of graphs motivated by the rank-width of graphs~\cite{OumS06}. 
	The \emph{vertex-minor} relation is a graph containment relation which was introduced by Bouchet~\cite{Bouchet1987a, Bouchet1987b, Bouchet88, Bouchet1988, Bouchet1989a} 
	in his studies of circle graphs and 4-regular Eulerian digraphs. 
	The vertex-minor relation has an important role in the theory of (linear) rank-width~\cite{Oum2004a, Oum05, Oum2006a, JKO2014, Oum12} as (linear) rank-width does not increase when taking vertex-minors of a graph. 
	We provide concise definitions in Section~\ref{sec:prelim}.
	
	The problem of computing linear rank-width has been discussed recently. 
	Kashyap~\cite{Kashyap2008} proved that it is NP-hard to compute matroid path-width on binary matroids.
	Proposition 3.1 in~\cite{Oum05} shows that the problem of determining the linear rank-width of a bipartite graph is equivalent to the problem of determining the path-width of a binary matroid, 
	and from this relation, we can show that computing linear rank-width is NP-hard in general.
	Adler and the authors of this paper~\cite{AdlerKK15} proved that the linear rank-width of distance-hereditary graphs, which are graphs of rank-width $1$, can be computed in time $\mathcal{O}(n^2\log n)$
	where $n$ is the number of vertices in an input graph. 
	Jeong, Kim, and Oum \cite{JeongKO16} showed that, there is a constructive algorithm to test whether a given graph has linear
rank-width at most $k$ in time $f(k)\cdot n^3$ for some function $f$. 
	Using this, they also proved that for every fixed integer $w$, there is a polynomial-time algorithm to compute linear rank-width on graphs of rank-width $w$.

	In this paper, we focus on structural aspects of linear rank-width.
		 The first result of the Graph Minor series papers is that for a fixed tree $T$, every graph of sufficiently large path-width contains a minor isomorphic to $T$~\cite{RobertsonS83}, 
	 and this was later used by Blumensath and Courcelle~\cite{BlumensathC10}
	to define a hierarchy of \emph{incidence graphs} based on \emph{monadic second-order transductions}.  In order to obtain a similar hierarchy for graphs, still based on monadic second-order
	transductions, Courcelle~\cite{CourcelleBanhoff08} asked whether for a fixed tree $T$, every bipartite graph of sufficiently large linear rank-width contains a vertex-minor isomorphic to $T$. 
	We conjecture that it is true for any graph. 
 	\begin{CON}\label{con:tree}
	For every fixed tree $T$, there is an integer $f(T)$ such that every graph of linear rank-width at least $f(T)$ contains a vertex-minor isomorphic to $T$.	
	\end{CON}

 	We show that Conjecture~\ref{con:tree} is true if and only if it is true in prime graphs with respect to \emph{split decompositions}~\cite{Cunningham1982}.  A \emph{split} in a graph is
        a vertex partition $(A,B)$ such that $\abs{ A}$, $\abs{B} \ge 2$ and the set of edges joining $A$ and $B$ induces a complete bipartite subgraph.  \emph{Prime graphs} are graphs without splits
        and they form, with complete graphs and stars, the basic graphs in the theory of canonical split decompositions developed by Cunningham~\cite{Cunningham1982}.  They are also considered when studying the rank-width of graphs because the
        rank-width of a graph is the maximum rank-width over all its prime induced subgraphs. 
        
        We prove the following.

\begin{theorem}\label{thm:largelrw}
	Let $p$ be a positive integer and let $T$ be a tree. Let $G$ be a graph such that every prime induced subgraph of $G$ has linear rank-width at most $p$.
	If $G$ has linear rank-width at least $40(p+2)\abs{V(T)}$, then $G$ contains a vertex-minor isomorphic to $T$.	
\end{theorem}

  A graph $G$ is \emph{distance-hereditary} if for every connected induced subgraph $H$ of $G$ and two vertices $v$ and $w$ in $H$, 
  the distance between $v$ and $w$ in $H$ is the same as their distance in $G$.
  It is known that every prime induced subgraph of a distance-hereditary graph has size at most $3$~\cite{Bouchet88}. 
 Together with this fact, our result implies that Conjecture~\ref{con:tree} is also true for distance-hereditary graphs.

	To prove Theorem~\ref{thm:largelrw}, we essentially prove that 
 	for a fixed tree $T$, every graph admitting a canonical split decomposition whose decomposition tree has sufficiently large path-width contains a vertex-minor isomorphic to $T$.
	Combined with a relation between the linear rank-width of a  graph and the path-width of its canonical split decomposition, we obtain Theorem~\ref{thm:largelrw}. We will obtain such a relation in Section~\ref{sec:pwofcanonicaltrees}.
 The vertex-minor relation cannot be replaced with the induced subgraph relation because there is a cograph admitting a canonical split decomposition whose decomposition tree has sufficiently large path-width~\cite{CorneilLB1981, GP2012}, but cographs have no $P_4$ as an induced subgraph. 

 In the second part, we investigate the set of distance-hereditary vertex-minor obstructions for graphs of bounded linear rank-width.
 A graph is a \emph{vertex-minor obstruction} for graphs of linear
   rank-width $k$ if it has linear rank-width $k+1$ and every proper vertex-minor has linear rank-width $k$.  
   Robertson and Seymour~\cite{RS2004} showed that for every infinite sequence
 $G_1, G_2, \ldots $ of graphs, there exist $G_i$ and $G_j$ with $i<j$ such that $G_i$ is isomorphic to a minor of $G_j$. In other words, graphs are \emph{well-quasi-ordered} under the minor
 relation. Interestingly, this property implies that for any proper class $\mathcal{C}$ of graphs closed under taking minors, the set of minor obstructions for $\mathcal{C}$ is finite.

Oum~\cite{Oum2004a, Oum12} obtained an analogous result for the vertex-minor relation;  
for every infinite sequence $G_1$, $G_2, \ldots$ of graphs of bounded rank-width, there exist $G_i$ and $G_j$ with $i<j$ such that $G_i$ is isomorphic to a vertex-minor of $G_j$. 
We can obtain the following as a corollary.
\begin{theorem}[Oum~\cite{Oum2004a}]\label{thm:vertexminorwqo}
For every class $\mathcal{C}$ of graphs with bounded
rank-width that is closed under taking vertex-minors, there is a finite list of graphs $G_1$, $G_2, \ldots, G_m$
such that a graph is in $\mathcal C$ if and only if it has no
vertex-minor isomorphic to $G_i$ for some $i\in \{1,2, \ldots, m\}$.
\end{theorem}
Theorem~\ref{thm:vertexminorwqo} implies that for every integer $k$, the class of all graphs of (linear) rank-width at most $k$ can be characterized by a finite list of vertex-minor obstructions. However, it does not give any explicit number of necessary vertex-minor obstructions or bound on the size of such graphs.
Oum~\cite{Oum05} proved that for each $k$, the size of a vertex-minor obstruction for graphs of rank-width at most $k$ is at most $(6^{k+1}-1)/5$.
	For linear rank-width, obtaining such an upper bound on the size of vertex-minor obstructions remains an open problem.
	Jeong, Kwon, and Oum~\cite{JKO2014} showed that  the number of vertex-minor obstructions for linear rank-width at most $k$ is at least $2^{\Omega(3^k)}$.

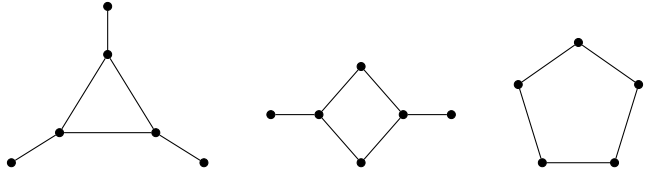
\begin{figure}
 \tikzstyle{v}=[circle, draw, solid, fill=black, inner sep=0pt, minimum width=3pt]
  \centering
     \begin{tikzpicture}[scale=0.8]
    \node[v](v1) at (0,.8){};
    \node[v](v2) at (0,1.6){};
    \node[v](v3) at (-.8,-.5){};
    \node[v](v4) at (.8,-.5){};
    \node[v](v5) at (-1.6,-1){};
    \node[v](v6) at (1.6,-1){};
      \draw (v1)--(v3)--(v4)--(v1);
      \draw (v1)--(v2);
      \draw (v3)--(v5);
      \draw (v4)--(v6);
    \end{tikzpicture}\quad     \quad
    \begin{tikzpicture}[scale=0.8]
    \node[v](v1) at (0,2){};
    \node[v](v2) at (1-.2,2){};
    \node[v](v3) at (1.5,1+.2){};
    \node[v](v4) at (1.5,3-.2){};
    \node[v](v5) at (2+.2,2){};
    \node[v](v6) at (3,2){};
      \draw (v1)--(v2)--(v3)--(v5)--(v6);
      \draw (v2)--(v4)--(v5);
    \end{tikzpicture}\quad\quad
      \begin{tikzpicture}[scale=0.8]
    \node[v](v1) at (1,3){};
    \node[v](v2) at (2,2.3){};
    \node[v](v3) at (1.6,1){};
    \node[v](v4) at (0.4,1){};
    \node[v](v5) at (0, 2.3){};
      \draw (v1)--(v2)--(v3)--(v4)--(v5)--(v1);
    \end{tikzpicture}  \caption{The three vertex-minor obstructions for graphs of linear rank-width at most $1$. The first two graphs are distance-hereditary. }
  \label{fig:vmobslrw1}
\end{figure}

	Adler, Farley, and Proskurowski~\cite{AdlerFP11} obtained the set of all three vertex-minor obstructions for graphs of linear rank-width at most $1$, depicted in Figure~\ref{fig:vmobslrw1}, 
	two of which are distance-hereditary.  
	In this paper, we construct a set of graphs containing all vertex-minor obstructions for graphs of linear rank-width at most $k$ that are distance-hereditary. 
	This is an analogous result to the characterization of acyclic minor obstructions for graphs of path-width at most $k$,
	investigated by Takahashi, Ueno, and Kajitani~\cite{TakahashiUK94}, and Ellis, Sudborough, and Turner~\cite{EllisST94}.
	As a similar work, Koutsonas, Thilikos, and Yamazaki~\cite{Thilikos2014}
	characterized matroid obstructions for bounded matroid path-width that are cycle matroids of outerplanar graphs.
	
	Lastly, we obtain simpler proofs of known characterizations of graphs of linear rank-width at most $1$~\cite{AdlerFP11,Bui-XuanKL13}.

	The paper is organized as follows.  
	Section~\ref{sec:prelim} provides some preliminary concepts, including linear rank-width and vertex-minors. 
	In Section~\ref{sec:splitdecs}, we introduce necessary notions regarding split decompositions, 
	and restate the structural characterization of linear rank-width on distance-hereditary graphs.
	Section~\ref{sec:pwofcanonicaltrees} presents a relation between the linear rank-width of a graph whose prime induced subgraphs have bounded linear rank-width 
	and the path-width of its decomposition tree. 
	From this, we prove Theorem~\ref{thm:largelrw} in Section~\ref{sec:treevertexminor}.
	In Section~\ref{sec:obstructions}, 
	we provide a way to generate all vertex-minor obstructions for graphs of bounded linear rank-width that are distance-hereditary graphs.  
	Section~\ref{sec:charlrw1} presents
	simpler proofs for known characterizations of the graphs of linear rank-width at most $1$.

\section{Preliminaries}\label{sec:prelim}

	In this paper, graphs are finite, simple and undirected.  Our graph terminology is standard, see for instance \cite{Diestel05}. 
	Let $G$ be a graph. We denote the vertex set of
	$G$ by $V(G)$ and the edge set by $E(G)$. 
	For $X\subseteq V(G)$, we denote by $G[X]$ the subgraph of $G$ induced
	by $X$, and let $G- X:=G[V(G)\setminus X]$. 
	For $v\in V(G)$, we write $G- x$ for $G- \{x\}$. 
	For $F\subseteq E(G)$, let $G-F:=(V(G), E(G)\setminus F)$.
	For a vertex $x$ of $G$, let $N_G(x)$ be the set of \emph{neighbors} of $x$ in
	$G$ and we call $\abs{N_G(x)}$ the \emph{degree} of $x$ in $G$.
	Two vertices $x$
	and $y$ are \emph{twins} if $N_G(x)\setminus \{y\}=N_G(y)\setminus \{x\}$. 
	An edge $e$ of a connected $G$ is a \emph{cut-edge} if $G-e$ is disconnected. 
	A vertex $v$ in a connected graph $G$ is a \emph{cut vertex} if $G-v$ is disconnected.
	A connected graph is \emph{$2$-connected} if it has at least $3$ vertices and has no cut vertices.

A \emph{tree} is a connected graph containing no cycles. 
A vertex of degree one in a tree is called a \emph{leaf}.  A \emph{subcubic tree} is a tree with maximum degree at most three, and 
a \emph{path} is a tree with maximum degree at most two.
The \emph{length} of a path is the number of its edges.  A \emph{star} is a tree with a distinguished vertex, called its \emph{center}, adjacent to all other vertices.  A \emph{complete graph} is a
graph with all possible edges.  A graph $G$ is called \emph{distance-hereditary} if for every pair of two vertices $x$ and $y$ of $G$ the distance of $x$ and $y$ in $G$ equals the distance of $x$ and
$y$ in any connected induced subgraph containing both $x$ and $y$~\cite{BandeltM86}. It is well-known that a graph is distance-hereditary if and only if it can be obtained from a single vertex by repeated addition of degree one vertices and twins \cite{HammerM90}. 
	An induced cycle of length at least $5$ is not distance-hereditary.

A subset $F$ of the edge set of $G$ is called a \emph{matching} if no two edges in $F$ share an end vertex.
 
For an edge $e$ of a graph $G$, we denote by $G/e$ the graph obtained by contracting $e$. 
A graph $H$ is a \emph{minor} of a graph $G$ if $H$ is obtained from a subgraph of $G$ by contractions of edges. 

\subsection{Linear rank-width}\label{subsec:lrw-vm}

For sets $R$ and $C$, an \emph{$(R,C)$-matrix} is a matrix whose rows and columns are indexed by $R$ and $C$, respectively.
  For an $(R,C)$-matrix $M$ and subsets  $X\subseteq R$ and $Y\subseteq C$, 
  let $M[X,Y]$ be the submatrix of $M$ whose rows and columns are indexed by $X$ and $Y$, respectively.

	Let $G$ be a graph. We denote by $A_G$ the \emph{adjacency matrix} of $G$ over the binary field; that is, for $v,w\in V(G)$, $A_G[v,w]=1$ if $v$ is adjacent to $w$, 
	and $A_G[v,w]=0$, otherwise. 
	For a graph $G$, let $\cutrk_G^*:2^{V(G)}\times 2^{V(G)} \to \mathbb{Z}$ be a function such that
	$\cutrk_G^*(X,Y):=\rank(A_G[X,Y])$ for all $X,Y\subseteq V(G)$, where rank is computed over the binary field.  
	The \emph{cut-rank function} of $G$ is the function $\cutrk_G:2^{V(G)}\rightarrow \mathbb{Z}$ where for each $X\subseteq V(G)$,
\[\cutrk_G(X):= \cutrk_G^*(X,V(G)\setminus X).\] 
	An ordering $(x_1, \ldots, x_n)$ of the vertex set $V(G)$ is called a \emph{linear layout} of $G$.  
	If $\abs{V(G)}\ge 2$, then the \emph{width} of a linear layout $(x_1,\ldots, x_n)$ of $G$ is defined
as
\[\max_{1\le i\le n-1}\{\cutrk_G(\{x_1,\ldots,x_i\})\},\]
and if $\abs{V(G)}=1$, then the width is defined to be $0$.
	The \emph{linear rank-width} of $G$, denoted by $\lrw(G)$, is defined as the minimum width over all linear layouts of $G$. 

	Caterpillars and complete graphs have linear rank-width at most $1$. Ganian~\cite{Ganian10} gave a characterization of 
	graphs of linear rank-width at most $1$, and called them \emph{thread graphs}. Adler and Kant\'{e}~\cite{AdlerK13} showed that 
	linear rank-width and path-width coincide on forests, and therefore, there is a linear-time algorithm to compute the linear rank-width of forests.  
	It is easy to see that the linear rank-width of a graph is the maximum over the linear
rank-widths of its connected components. 

	For a linear layout $L$ of a graph $G$ and $v,w\in V(G)$, 
	we denote $v\le_L w$ if $v=w$ or $v$ appears before $w$ in the linear layout. 
	For two orderings $(v_1, v_2, \ldots, v_n)$ and $(w_1, w_2, \ldots, w_m)$, 
	we denote 
	\[(v_1, v_2, \ldots, v_n)\oplus (w_1, w_2, \ldots, w_m):=(v_1, v_2, \ldots, v_n, w_1, w_2, \ldots, w_m).\]

\subsection{Vertex-minors}  
	For a graph $G$ and a vertex $x$ of $G$, the \emph{local complementation at $x$} in $G$ is an operation to replace the subgraph induced 
	by the set of neighbors of $x$ with its complement. 
	The resulting graph is denoted by $G*x$.  If a graph $H$ can be obtained from $G$ by applying a sequence of local complementations, then $G$ and $H$ 
	are called \emph{locally equivalent}.  A graph $H$ is called a
	\emph{vertex-minor} of a graph $G$ if $H$ can be obtained from $G$ by applying a sequence of local complementations and deletions of vertices.
	Bouchet~\cite{Bouchet1989a} observed that local complementation does not change the cut-rank function. 
	This directly implies that 
	every vertex-minor 
   $H$ of $G$ satisfies that $\lrw(H) \leq \lrw(G)$. 
\begin{lemma}[Bouchet~\cite{Bouchet1989a}; See Corollary 2] \label{lem:vm-rw} 
Let $G$ be a graph and let $x$ be a vertex of $G$. Then for every subset $X$ of $V(G)$, we have $\cutrk_G(X)=\cutrk_{G*x}(X)$. 
\end{lemma}

For an edge $xy$ of $G$, let $W_1:=N_G(x)\cap N_G(y)$, $W_2:=(N_G(x)\setminus N_G(y))\setminus \{y\}$, and $W_3:=(N_G(y)\setminus N_G(x))\setminus \{x\}$.  The \emph{pivoting on $xy$} of $G$,
denoted by $G\wedge xy$, is the operation to flip the adjacencies between distinct sets $W_i$ and $W_j$, and swap the vertices $x$ and $y$. Flipping the adjacency between two vertices $v$ and $w$ is an operation that add an edge if there was no edge between $v$ and $w$, and remove an edge, otherwise. It is known that $G\wedge xy=G*x*y*x=G*y*x*y$ \cite[Proposition 2.1]{Oum05}. See Figure~\ref{fig:pivotex} for an example.

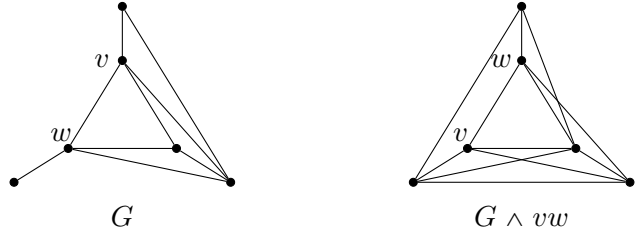
\begin{figure}
 \tikzstyle{v}=[circle, draw, solid, fill=black, inner sep=0pt, minimum width=3pt]
  \centering
     \begin{tikzpicture}[scale=0.9]
    \node[v](v1) at (0,.8){};
    \node[v](v2) at (0,1.6){};
    \node[v](v3) at (-.8,-.5){};
    \node[v](v4) at (.8,-.5){};
    \node[v](v5) at (-1.6,-1){};
    \node[v](v6) at (1.6,-1){};
      \draw (v1)--(v3)--(v4)--(v1);
      \draw (v1)--(v2);
      \draw (v1)--(v6);
      \draw (v3)--(v6);
      \draw (v2)--(v6);
      \draw (v3)--(v5);
      \draw (v4)--(v6);
      \draw (-0.3,.8) node{$v$};
      \draw (-0.6-0.3,-.3) node{$w$};
      \draw (0,-1.5) node{$G$};
    \end{tikzpicture}\qquad\qquad\qquad
         \begin{tikzpicture}[scale=0.9]
    \node[v](v1) at (0,.8){};
    \node[v](v2) at (0,1.6){};
    \node[v](v3) at (-.8,-.5){};
    \node[v](v4) at (.8,-.5){};
    \node[v](v5) at (-1.6,-1){};
    \node[v](v6) at (1.6,-1){};
      \draw (v1)--(v3)--(v4)--(v1);
      \draw (v1)--(v2);
      \draw (v3)--(v5);
      \draw (v4)--(v6);
      \draw (v1)--(v6);
      \draw (v3)--(v6);
      
      \draw (v2)--(v4);
      \draw (v5)--(v4);
      \draw (v5)--(v6);
      \draw (v2)--(v5);
      \draw (-0.3,.8) node{$w$};
      \draw (-0.6-0.3,-.3) node{$v$};
      \draw (0,-1.5) node{$G\wedge vw$};
    \end{tikzpicture}\caption{An example of pivoting.}
  \label{fig:pivotex}
\end{figure}

\subsection{Path-width}

A \emph{path decomposition} of a graph $G$ is a pair $(P,\cB)$, where $P$ is a path and $\cB=(B_t)_{t\in V(P)}$ is a family of vertex subsets of $G$ such that 
\begin{enumerate}
	\item for every $v\in V(G)$ there exists $t\in V(P)$ such that $v\in B_t$,
	\item for every $uv\in E(G)$ there exists $t\in V(P)$ such that $\{u,v\}\subseteq B_t$,
	\item for every $v\in V(G)$, the set $\{t\in V(P): v\in B_t\}$ induces a subpath of $P$.
\end{enumerate}
The \emph{width} of a path decomposition $(P,\cB)$ is defined as $\max\{\abs{B_t}: t\in V(P)\}-1$.
The \emph{path-width} of $G$, denoted by $\pw(G)$, is defined as the  minimum width over all path-decompositions of $G$.

	It is well known that if $H$ is a minor of $G$, then $\pw(H)\leq \pw(G)$. Robertson and Seymour~\cite{RobertsonS83} first proved that for a fixed tree $T$, every 
	graph of sufficiently large path-width contains a minor isomorphic to $T$. The necessary function was optimized by Bienstock, Robertson, Seymour, and Thomas~\cite{BienstockRST91}.

\begin{theorem}[Bienstock, Robertson, Seymour, and Thomas~\cite{BienstockRST91}]\label{thm:pathwidththeorem} For every forest $F$, every graph with path-width at least $\abs{V(F)}-1$ has a minor isomorphic to $F$. \end{theorem} 

	We recall the following theorem which characterizes the path-width of trees and is used for computing their path-width in linear time.

\begin{theorem}[Ellis, Sudborough, and Turner~\cite{EllisST94}; Takahashi, Ueno, and Kajitani~\cite{TakahashiUK94}]\label{thm:maintreeforpathwidth}
Let $T$ be a tree and let $k$ be a positive integer. The following are equivalent.
\begin{enumerate}[(1)]
\item $T$ has path-width at most $k$.
\item For every node $x$ of $T$, at most two of the subtrees of $T- x$ have path-width $k$ and all other subtrees of $T- x$ have path-width at most $k-1$.
\item $T$ has a path $P$ such that for each node $v$ of $P$ and each connected component $T'$ of $T- v$ not containing a node of $P$, $\pw(T')\le k-1$.
\end{enumerate}
\end{theorem}

\section{Linear rank-width of distance-hereditary graphs}\label{sec:splitdecs} 

	In this section, we recall the characterization of the linear rank-width of distance-hereditary graphs investigated by Adler and the authors of this paper~\cite{AdlerKK15}.  
	For this characterization, we need to introduce split decompositions and the new notion of limbs introduced in the previous paper.
	We will follow the definition for split decompositions used by Bouchet~\cite{Bouchet88}. 

	A \emph{split} in a connected graph $G$ is a vertex partition $(X,Y)$ of $G$ such that $|X|,|Y|\geq 2$ and $\cutrk_G(X)=1$. 
	\emph{Prime graphs} are connected graphs that do not have a split. 
	Note that every connected graph with at most $3$ vertices is a prime graph, by definition.
	Also, one can observe that every connected graph on $4$ vertices admits a split, and it is not a prime graph.

	A \emph{marked graph} is a connected graph $D$ with a matching $M(D)$ 
	where every edge in $M(D)$ is a cut-edge. 
	Every edge in $M(D)$ is called a marked edge, and 
	the end vertices of marked edges are called \emph{marked vertices}. 
	The connected components of $D-M(D)$ are called \emph{bags} of $D$. 
	The edges in $E(D)\setminus M(D)$ are called \emph{unmarked edges}, 
	and the vertices that are not marked are called \emph{unmarked vertices}. 

	If $(X,Y)$ is a split in a marked graph $G$, then we construct a new marked graph $D$ such that 
	\begin{itemize}
	\item $V(D)=V(G) \cup \{x',y'\}$ for two distinct new vertices $x',y'\notin V(G)$,  
	\item $E(D)=E(G[X]) \cup E(G[Y]) \cup \{x'y'\} \cup E'$ where
	\begin{align*}
	E' &:= \{x'x: x\in X\ \textrm{and there exists $y\in Y$ such that $xy\in E(G)$}\} \cup\\ & 
	\qquad \{y'y : y\in Y\ \textrm{and there exists $x\in X$ such 	that $xy\in E(G)$}\},
	\end{align*}
	\item $x'y'$ is a marked edge, and all edges in $E'$ are unmarked edges.
	\end{itemize}
	The marked graph $D$ is called a \emph{simple decomposition of} $G$.
	See Figure~\ref{fig:simple} for an example.

\begin{figure}
\centerline{
\includegraphics[scale=0.55]{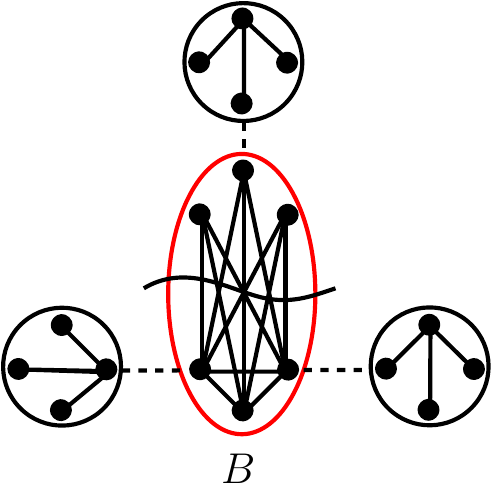} \quad\quad
\includegraphics[scale=0.55]{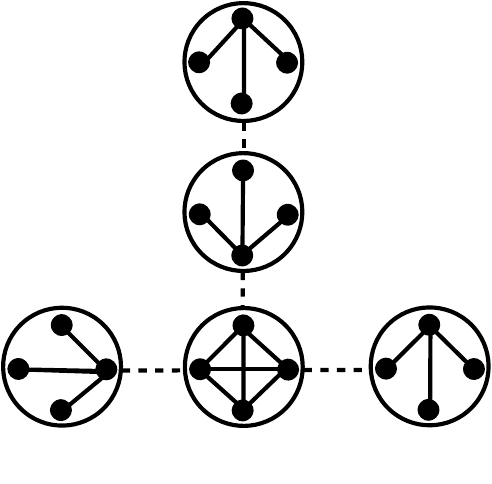} }
\caption{An example of replacing a bag $B$ with its simple decomposition. Circles indicate bags and dotted edges indicate marked edges. }
\label{fig:simple}
\end{figure}

	A \emph{split decomposition} of a connected graph $G$ is a marked graph $D$ defined inductively to be either $G$ or a marked graph
	defined from a split decomposition $D'$ of $G$ by replacing a bag with its simple decomposition. 
	For a marked edge $xy$ of a marked graph $D$, the \emph{recomposition
  of $D$ along $xy$} is the marked graph $(D\wedge xy) -\{x,y\}$.  For a split decomposition $D$, let $\origin{D}$ denote the graph obtained from $D$ by recomposing all marked
edges. Note that if $D$ is a split decomposition of $G$, then $\origin{D}=G$.  

	Since each marked edge of a split decomposition $D$ is a cut-edge and all marked edges form a matching, 
	if we contract all unmarked edges in $D$, then we obtain a tree. 
	We call it the \emph{decomposition tree of $G$ associated with $D$} and denote it by $T_D$.  
	To distinguish the vertices of $T_D$ from the vertices of $G$ or $D$, 
	the vertices of $T_D$ will be called \emph{nodes}. 
   For a node $v$ of $T_D$, we write $\bag{D}{v}$ to
   denote the bag of $D$ with which it is in correspondence, 
   and for a bag $B$ of $D$, we write $\node{D}{B}$ to 
   denote the node of $T_D$ with which it is in correspondence.
	Two bags of $D$ are
	called \emph{adjacent bags} if their corresponding nodes in $T_D$ are adjacent.
	A sequence of bags $B_1-B_2- \cdots - B_m$ is called a path of bags if for each $i\in \{1, 2, \ldots, m-1\}$, 
	$B_i$ and $B_{i+1}$ are adjacent bags, and all of $B_1, B_2, \ldots, B_m$ are pairwise distinct.
	Clearly, for two bags $B$ and $B'$, there is a unique path of bags from $B$ to $B'$, 
	which corresponds to the path from $\node{D}{B}$ to $\node{D}{B'}$ in $T_D$.
	We denote by $\dist_D(B, B')$ the distance from $\node{D}{B}$ to $\node{D}{B'}$ in $T_D$; in other words, it is one less than the number of 
	bags in the unique path of bags from $B$ to $B'$ in $D$.
	
	\subsection{Canonical split decompositions and local complementations}\label{subsec:canonical}
	A split decomposition is called \emph{canonical} if each bag is either a prime graph, a star, or a complete graph, and 
	every recomposition of a marked edge in D results in a split decomposition without the same property.  
	The following is due to Cunningham and Edmonds \cite{CunninghamE80}, and Dahlhaus \cite{Dahlhaus00}.

	\begin{theorem}[Cunningham and Edmonds~\cite{CunninghamE80}; Dahlhaus~\cite{Dahlhaus00}] \label{thm:CED} Every connected graph $G$ has a unique canonical split decomposition, up to 	isomorphism, and it can be computed in time $\cO(|V(G)|+|E(G)|)$.
\end{theorem}

	A bag is called a \emph{prime bag} if it is a prime graph on at least $5$ vertices, and a bag is called a complete bag or a star bag if it is a complete graph or a star, respectively. 

	Let $D$ be a split decomposition of a connected graph $G$ with bags that are either a prime graph, a complete graph or a star.
	The \emph{type of a bag} of $D$ is either $P$, $K$, or $S$ depending on whether it is a prime graph, a complete graph, or a star, respectively. 
	The \emph{type of a marked edge} $uv$  is $AB$ where $A$ and $B$ are the types of the bags containing $u$ and $v$ respectively. If $A=S$ or $B=S$, then we can replace
$S$ by $S_p$ or $S_c$ depending on whether the end of the marked edge is a leaf or the center of the star, respectively. 
Bouchet characterized when it becomes a canonical split decomposition.

\begin{theorem}[Bouchet~\cite{Bouchet88}]\label{thm:can-forbid} Let $D$ be a split decomposition of a connected graph whose bags are either a prime graph, a complete graph, or a star. Then $D$ is a canonical split decomposition if and only if it has no marked edge of type $KK$ or   $S_pS_c$.
\end{theorem}

 We will use the following characterizations of trees and of distance-hereditary graphs.

\begin{theorem}[Bouchet~\cite{Bouchet88}]\label{thm:Bouchet88}\hfill \begin{enumerate} [(1)]
\item A connected graph is distance-hereditary if and only if every bag of its canonical split decomposition is of type K or S.  
\item A connected graph is a tree if and only if every bag of its canonical split decomposition is a star bag whose center is an unmarked vertex.
\end{enumerate}
\end{theorem}

	We now relate the split decompositions of a graph and the ones of its locally equivalent graphs. 
	Let $D$ be a split decomposition of a connected graph.  A vertex $v$ of $D$ \emph{represents} an unmarked vertex $x$ (or is a \emph{representative} of $x$) 
	if either $v=x$ or there is a path of even length from $v$ to $x$ in $D$ starting with a marked edge such that marked edges and unmarked edges 
	appear alternately in the path.  
	Two unmarked vertices $x$ and $y$ are \emph{linked} in $D$ if there is a path from $x$ to $y$ in $D$ 
	such that unmarked edges and marked edges appear alternately in the path.
	Linkedness of unmarked vertices exactly represents the adjacency relation between those vertices in the original graph.
	 
\begin{lemma}[Adler, Kant\'e, and Kwon~\cite{AdlerKK15}] \label{lem:represent} Let $D$ be a split decomposition of a connected graph $G$.  Let $v'$ and $w'$ be two vertices in a same bag of $D$, and let $v$ and $w$ be two
  unmarked vertices of $D$ represented by $v'$ and $w'$, respectively.  The following are equivalent.
 \begin{enumerate}
 \item $v$ and $w$ are linked in $D$.
 \item $vw\in E(G)$.
 \item $v'w' \in E(D)$.
 \end{enumerate}
\end{lemma}

	A \emph{local complementation} at an unmarked vertex $x$ in a split decomposition $D$, denoted by $D*x$, 
	is the operation to replace each bag $B$ containing a representative $w$ of $x$
	with $B*w$. Bouchet observed that $D*x$ is a split decomposition of $\origin{D}*x$, and $M(D) = M(D*x)$.
	Two split decompositions $D$ and $D'$ are \emph{locally equivalent}
	if $D$ can be obtained from $D'$ by applying a sequence of local complementations at unmarked vertices.
	As expected, this local complementation also preserves the property that the split decomposition is canonical.
	
	\begin{figure}
\centerline{
\includegraphics[scale=0.45]{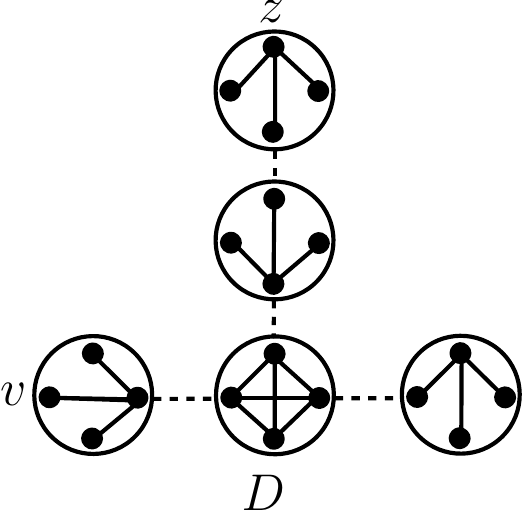} \quad\quad
\includegraphics[scale=0.45]{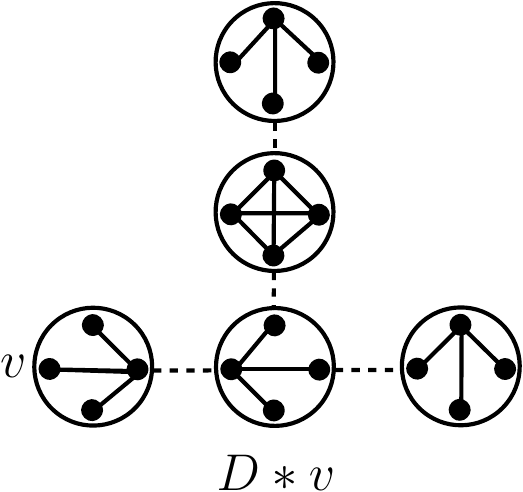} \quad\quad
\includegraphics[scale=0.45]{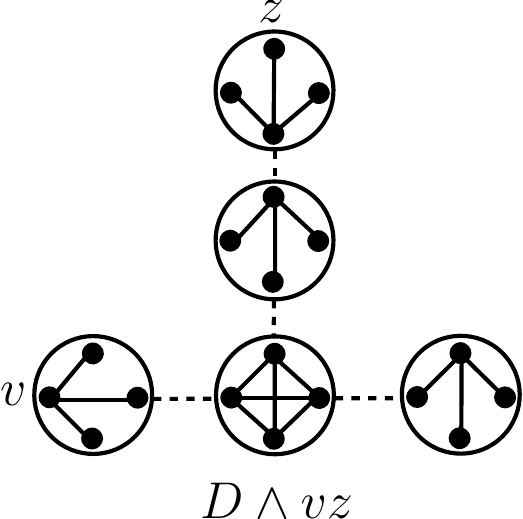}}
\caption{Examples of local complementation and pivoting in a split decomposition. }
\label{fig:local}
\end{figure}
	
\begin{lemma}[Bouchet~\cite{Bouchet88}]\label{lem:localdecom} 
	Let $D$ be the canonical split decomposition of a connected graph $G$. 
	If $x$ is a vertex of $G$, then $D*x$ is the canonical split decomposition of $G*x$.
\end{lemma}

	Let $x$ and $y$ be linked unmarked vertices in a split decomposition $D$, and let $P$ be the path in $D$ linking $x$ and $y$
	such that unmarked edges and marked edges appear alternately in the path.  
	Note that if $B$ is a bag of type $S$ containing an unmarked edge of $P$, then the center of $B$ is a representative of either $x$ or $y$.  The \emph{pivoting 	on $xy$ of $D$}, denoted by $D\wedge xy$, is
	the split decomposition obtained as follows: for each bag $B$ containing an unmarked edge of $P$, if $v, w\in V(B)$ represent 
	respectively $x$ and $y$ in $D$, then we replace $B$ with $B\wedge vw$. It is
	worth noticing that by Lemma~\ref{lem:represent}, we have $vw\in E(B)$, hence $B\wedge vw$ is well-defined.

\begin{lemma}[Adler, Kant\'e, and Kwon~\cite{AdlerKK15}] \label{lem:pivotdecom} 
	Let $D$ be a split decomposition of a connected graph $G$. 
	If $xy\in E(G)$, then $D\wedge xy=D*x*y*x$.  
\end{lemma}

	\subsection{Removing vertices}\label{subsec:modifications}
	
	Let $G$ be a distance-hereditary graph and let $D$ be its split decomposition. 
	Let $S$ be a vertex set of $G$.
	We explain how we transform $D$ into a split decomposition of $G-S$.
	Note that the split decomposition obtained from $D$ by removing vertices in $S$ is not necessarily a split decomposition 
	because the resulting marked graph may have bags of size at most $2$.
	In this case, we need to recompose a marked edge incident with each bag of size at most $2$ unless the resulting marked graph has at most two vertices.
	
	Suppose $D$ is canonical.
	We frequently consider connected components $T$ of $D-V(B)$, for a bag $B$ of $D$.
	This will be used to define limbs in the next subsection. 
	For a bag $B$ of $D$ and a connected component $T$ of $D-V(B)$, let us denote by $\zeta_b(D,B,T)$ and $\zeta_c(D,B,T)$ 
	the end vertices of the marked edge in $D$ linking $B$ and $T$ that are in $V(B)$ and in $V(T)$ respectively.
	Subscripts $b$ and $c$ stand for bag and component, respectively. 
    We always treat $T$ as a canonical split decomposition and regard $\zeta_c(D,B,T)$ as an unmarked vertex.

	\subsection{Limbs and characterization of linear rank-width}\label{subsec:limbs}

	To present the characterization of the linear rank-width of distance-hereditary graphs, 
	we need the new notion called limbs~\cite{AdlerKK15}.
	For an unmarked vertex $y$ in $D$ and a bag $B$ of $D$ containing a marked vertex representing $y$,
	let $T$ be the connected component of $D-V(B)$ containing $y$, and 
	let $v:=\zeta_c(D,B,T)$ and $w:=\zeta_b(D,B,T)$.
	We define the \emph{limb} $\limb:=\limb_D[B,y]$ with respect to $B$ and $y$ as follows:
	\begin{enumerate}
	\item if $B$ is of type $K$, then $\limb:=T*v- v$,       
	\item if $B$ is of type $S$ and $w$ is a leaf, then $\limb:=T- v$,
	\item if $B$ is of type $S$ and $w$ is the center, then $\limb:=T\wedge vy - v$.
	\end{enumerate}
	While $T$ is a canonical split decomposition, $\limb$ may not be a canonical split decomposition, because deleting $v$ may create a
	bag of size $2$.  We analyze the cases when such a bag appears, and describe how to transform it into a canonical split decomposition.
	Suppose that a bag $B'$ of size $2$ appears in $\limb$.
	If $B'$ has no adjacent bags in $\limb$, then $B'$ itself is a canonical split decomposition.
	We may assume there is a bag adjacent to $B'$.

	\begin{figure}
\centerline{
\includegraphics[scale=0.45]{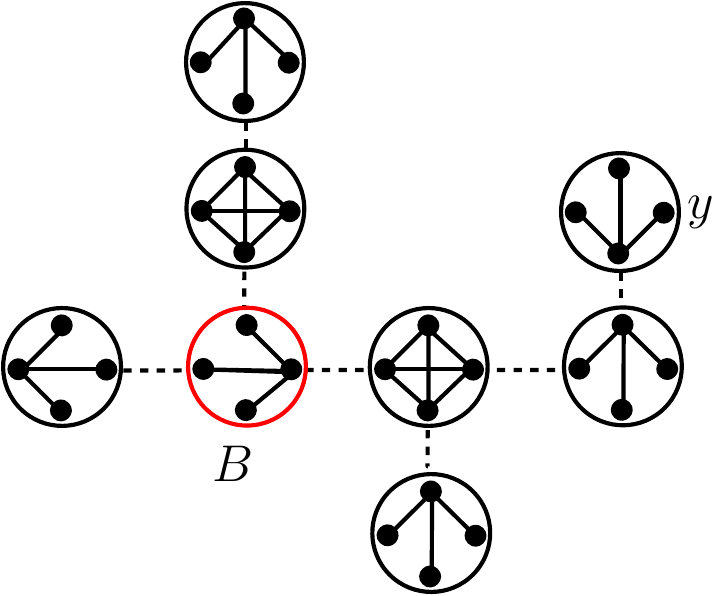} \quad\quad
\includegraphics[scale=0.45]{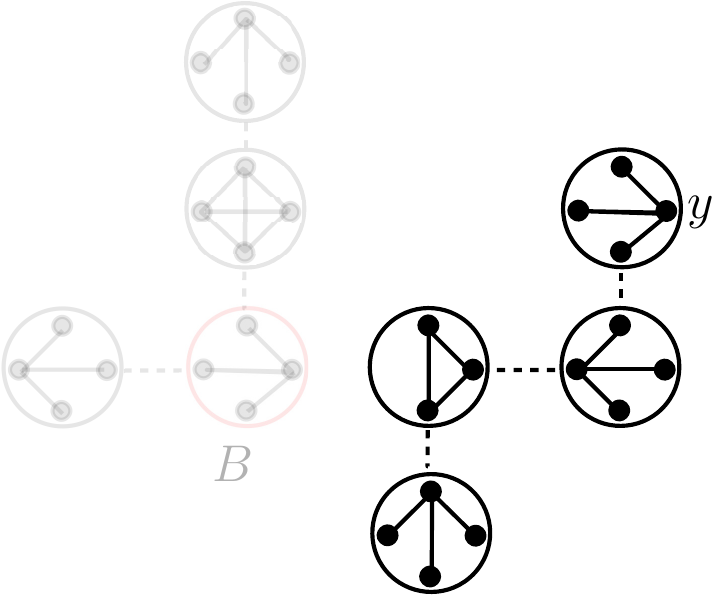}}
\caption{An example of a limb $\mathcal{L}_D[B, y]$. }
\label{fig:limb}
\end{figure}

	\begin{enumerate}
	\item ($B'$ has one adjacent bag $B_1$.) \\
	If $v_1\in V(B_1)$ is the marked vertex adjacent to a vertex of $B'$	
	and $r$ is the unmarked vertex of $B'$ in $\mathcal{L}$,
	then we remove the bag $B'$ and replace $v_1$ with $r$. In other words, we recompose along the marked edge connecting $B'$ and $B_1$.

\item ($B'$ has two adjacent bags $B_1$ and $B_2$.)\\
  	If $v_1\in V(B_1)$ and $v_2\in V(B_2)$ are the two marked vertices that are adjacent to the two marked vertices of $B'$, then we 
  	remove $B'$ and add a marked edge $v_1v_2$. If the new marked edge $v_1v_2$ is of type KK or $S_pS_c$, then by recomposing along $v_1v_2$, we finally transform the limb into a canonical split decomposition.
\end{enumerate}

	Let $\limbtil_D[B,y]$ be the canonical split decomposition obtained from $\limb_D[B,y]$ and we call it the \emph{canonical limb}.  
	Let $\limbhat_D[B,y]$ be the graph obtained from $\limb_D[B,y]$ by
	recomposing all marked edges. 
	For a bag $B$ of $D$ and a connected component $T$ of $D- V(B)$, we define $f_D(B,T)$ as the linear rank-width of $\limbhat_D[B,y]$ for some unmarked vertex $y\in V(T)$.  
	It was shown that  $f_D(B,T)$ does not depend on the choice of $y$. 
	
	 \begin{proposition}[Adler, Kant\'e, and Kwon; Proposition 3.4 of \cite{AdlerKK15}]\label{prop:preservelrw} 
Let $B$ be a bag of $D$ and let $y$ be an unmarked vertex of $D$ represented by a vertex $w$ in $B$.
 Let $x\in V(\origin{D})$.  
 If an unmarked vertex $y'$ is represented by $w$ in $D*x$, then $\limbhat_D[B,y]$ is locally equivalent to $\limbhat_{D*x}[(D*x)[V(B)],y']$. Therefore,
   $f_D(B,T)=f_{D*x}((D*x)[V(B)],T_x)$ where $T$ and $T_x$ are the components of $D\setminus V(B)$ and $(D*x)\setminus V(B)$ containing $y$, respectively.
 \end{proposition}

	As a variant of Theorem~\ref{thm:maintreeforpathwidth}, distance-hereditary graphs of bounded linear rank-width can be characterized using limbs.
      
\begin{theorem}[Adler, Kant\'e, and Kwon~\cite{AdlerKK15}]\label{thm:mainchap2}  
  Let $k$ be a positive integer and let $D$ be the canonical split decomposition of a connected distance-hereditary graph $G$.  Then the following are equivalent. 
  \begin{enumerate}[(1)]
  \item $G$ has linear rank-width at most $k$.
  \item For each bag $B$ of $D$, $D-V(B)$ has at most two connected components $T$ such that $f_D(B,T)=k$, and every other connected component $T'$ of $D- V(B)$ satisfies that $f_D(B,T')\le k-1$.
  \item $T_D$ has a path $P$ such that for each node $v$ of $P$ and each connected component $H$ of $D-V(\bag{D}{v})$ containing no bags $\bag{D}{w}$ with $w\in V(P)$,  $f_D(\bag{D}{v},H)\le k-1$.
  \end{enumerate}
\end{theorem}

\section{Path-width of decomposition trees}\label{sec:pwofcanonicaltrees}

	To prove Theorem~\ref{thm:largelrw}, 
	we derive a relation between the linear rank-width of a graph whose prime induced subgraphs have bounded linear rank-width 
	and the path-width of its decomposition tree. 

\begin{proposition}\label{prop:generalupperbound}
  Let $p$ be a positive integer. Let $G$ be a connected graph whose prime induced subgraphs have linear rank-width at most $p$, and let $D$ be the canonical split decomposition of $G$, and let $T_D$ be the decomposition tree of $G$ associated with $D$.  
  Then $\lrw(G) \le 2(p+2)(\pw(T_D)+1)$.
\end{proposition}	

	We prove Proposition~\ref{prop:generalupperbound} by induction on the path-width of $T_D$.
	If its path-width is $0$, then it consists of one node, and the result directly follows from the given condition that every prime induced subgraph has linear rank-width at most $p$.
	Note that complete graphs and stars have linear rank-width at most $1$.
	We assume that the path-width of $T_D$ is at least $1$. 
	Using Lemma~\ref{thm:maintreeforpathwidth},
	$T$ contains a path $P$ such that for each node $v$ of $P$ and each connected component $T'$ of $T- v$ not containing a node of $P$, $\pw(T')\le k-1$.
	So, by induction, we can obtain an upper bound of the linear rank-width of split decompositions corresponding to such components $T'$.
	From this, we will obtain an upper bound of the linear rank-width of the whole graph.
		
	We need the following lemma. We point out that Lemma~\ref{lem:genupperbound} does not require 
	$D$ to be a canonical split decomposition, and this relaxation will be useful for an easier argument in the main proof.

\begin{lemma}\label{lem:genupperbound} 
	Let $k$ and $p$ be positive integers. 
	Let $B$ be a bag of a split decomposition $D$ with two unmarked vertices $a$ and $b$ such that 
	for every connected component $H$ of $D- V(B)$, $\lrw(\origin{H})\le k$.
	If $B$ has a linear layout of width at most $p$ 
	whose first and last vertices are $a$ and $b$ respectively, then $\origin{D}$ has a linear
  	layout of width at most $2p+k$ whose first and last vertices are $a$ and $b$ respectively.
\end{lemma}

\begin{proof} 
	Let $G:=\origin{D}$, and 
	let $L_B:=( w_1, w_2, \ldots, w_m)$ be a linear layout of $B$ of width at most $p$ such that $a=w_1$ and $b=w_m$. 
	For each $j\in \{1, 2, \ldots, m\}$, 
	\begin{enumerate}
	\item if $w_j$ is an unmarked vertex, then let $L_j:=(w_j)$, and
	\item if $w_j=\zeta_b(D, B, H)$ for some connected component $H$ of $D-V(B)$, 
	then let $L_j$ be a linear layout of
	$\origin{H}-\zeta_c(D, B, H)$ having width at most $k$.
	\end{enumerate}
	We define $L:=L_1\oplus L_2\oplus \cdots \oplus L_m$.
	We observe that $L$ is a linear layout of $G$.
	For each $j\in \{1, 2, \ldots, m\}$, we choose an unmarked vertex $y_j$ represented by $w_j$.
	If $w_j$ is an unmarked vertex, then $y_j=w_j$.

        We claim that $L$ has width at most $2p+k$.  It is sufficient to prove that for every $w\in V(G)\setminus \{a,b\}$, 
        $\cutrk_{G}(\{v: v\le_{L} w\})\le 2p +k$.  Let
        $w\in V(G)\setminus \{a,b\}$ and let $S_w:=\{v:v\le_{L} w\}$ and $T_w:=V(G)\setminus S_w$.  

	Let $H_j$ be a connected component of $D-V(B)$ such that $\zeta_b(D, B, H_j)=w_j$.
	Observe that if all vertices in $V(H_j)\cap V(G)$ are contained in $S_w$, 
	then all vertices in $V(H_j)\cap V(G)$ that have a neighbor in $T_w$ have exactly the same set of neighbors in $T_w$, which is $N_G(y_j)\cap T_w$.
	Therefore, when we compute the rank of the matrix $A(G)[S_w, T_w]$, 
	we can replace all vertices in $V(H_j)\cap V(G)$ with $y_j$. 
	The same observation holds for connected components fully contained in $T_w$.
	Also, for two distinct connected components $H_{j_1}, H_{j_2}$ of $D-V(B)$ where
	all vertices of $V(H_{j_1})\cap V(G)$ are contained in $S_w$ and all vertices of $V(H_{j_2})\cap V(G)$ are contained in $T_w$, 
	$y_1$ and $y_2$ are adjacent in $G$ if and only if 
	$\zeta_b(D, B, H_1)$ is adjacent to $\zeta_b(D, B, H_2)$ in $B$.
	This is an implication of Lemma~\ref{lem:represent}.

        Having it, we can observe that if $w$ is an unmarked vertex in $B$, then
	\[\cutrk_{G}(S_{w})=\cutrk_{B}(\{v: v\le_{L_{B}} w\})\le p.\] 
	Thus, we may assume that $w$ is contained in some connected component $H$ of $D-V(B)$.
	Let $j\in \{1, 2, \ldots, m\}$ such that $\zeta_b(D, B, H)=w_j$.

	Note that $H$ is the unique component of $D-V(B)$ possibly intersecting both $S_w$ and $T_w$.
	Since all vertices of $V(H)\cap V(G)$ having a neighbor in $V(G)\setminus V(H)$ have the same neighborhood in $V(G)\setminus V(H)$ (that is, $(V(H)\cap V(G), V(G)\setminus V(H))$ is a split), 
	we have 
	\begin{enumerate}[(1)]
	\item $\cutrk^*_{G}(S_{w}, T_{w}\setminus  V(H))
	\le 
	\max \{\cutrk_{B}(\{v: v\le_{L_{B}} w_{j-1}\}), \cutrk_{B}(\{v: v\le_{L_{B}} w_j\})\}
	\le p$.
	\item 
	$\cutrk^*_{G}(S_{w}\setminus V(H), T_{w})
	\le 
	\max \{\cutrk_{B}(\{v: v\le_{L_{B}} w_{j-1}\}), \cutrk_{B}(\{v: v\le_{L_{B}} w_j\})\} 
	\le p. $
	\item  $\cutrk^*_{G}(S_{w}\cap V(H), T_{w}\cap V(H))\le k$.
	\end{enumerate}
	Therefore, we have
	\begin{align*}
	\cutrk_{G}(S_{w})
	&\le \cutrk^*_{G}(S_{w}, T_{w}\setminus  V(H))
	+\cutrk^*_{G}(S_{w}\setminus V(H), T_{w}) \\ 
	&+ \cutrk^*_{G}(S_{w}\cap V(H), T_{w}\cap V(H))\\
	&\le p + p + k\le 2p+k.
	\end{align*}
      
        We conclude that $L$ is a linear layout of $G$ of width at most $2p+k$ whose first and last vertices are $a$ and $b$, respectively.
        \end{proof}
	
\begin{proof}[Proof of Proposition~\ref{prop:generalupperbound}]
	We prove it by induction on $k:=\pw(T_D)$.
	If $k=0$, then $T_D$ consists of one node, and $G$ is either a prime graph, a complete graph, or a star.
	Note that complete graphs and stars have linear rank-width at most $1$. 
	Thus, we have $\lrw(G)\le p\le 2(p+2)$.
	We may assume that $k\ge 1$.

      Since $\pw(T_D)=k\ge 1$, by Theorem~\ref{thm:maintreeforpathwidth}, there exists a path $P:=v_1v_2 \cdots v_n$ in $T_D$ such that for each node $v$ in $P$ and each connected component $T$
        of $T_D- v$ not intersecting $P$, $\pw(T)\le k-1$.  For each $i\in \{1, 2, \ldots, n\}$, let $B_i:=\bag{D}{v_i}$.  By induction hypothesis, for each $i\in \{1, 2, \ldots, n\}$ and each connected component $H$ of
        $D-V(B_i)$ not intersecting $\bigcup_{1\leq j \leq n} V(B_j)$, we have $\lrw(\origin{H}) \leq 2(p+2)k$.
	
	Now, let us  modify the given canonical split decomposition by two additional unmarked vertices so that
	we can easily apply Lemma~\ref{lem:genupperbound}. 
	For each $i\in \{1, 2, \ldots, n\}$, let $L_{B_i}$ be a linear layout of $B_i$ of width at most $p$.	
	First, we add a twin of the first vertex of $L_{B_1}$ in $B_1$ such that the added vertex is unmarked.
	Similarly, we add a twin of the last vertex of $L_{B_n}$ in $B_n$ such that the added vertex is unmarked.
	Let $a_1$ be the vertex added to $B_1$ and $b_n$ be the vertex added to $B_n$.
	It is not difficult to see that $B_1$ has a linear layout of width at most $p$ whose first vertex is $a_1$, and 
	$B_n$ has a linear layout of width at most $p$ whose last vertex is $b_n$.
	
	Assume for a moment that $n\ge 2$.
	 For each $i\in \{1, 2, \ldots, n-1\}$, 
	 let $b_i$ and $a_{i+1}$ be the marked vertices of $B_i$ and $B_{i+1}$, respectively, such that 
	 $b_ia_{i+1}$ is the marked edge connecting $B_i$ and $B_{i+1}$. 
	 If $b_i$ is not the end vertex of $L_{B_i}$, then we reorder $L_{B_i}$ so that $b_i$ is the end vertex.
	 Similarly, if $a_{i+1}$ is not the first vertex of $L_{B_{i+1}}$, then we reorder $L_{B_{i+1}}$ so that $a_{i+1}$ is the first vertex.
	 Until now, the width of each $L_{B_i}$ may increase by at most $2$.
	 This is because the rank of a matrix increase by at most $1$ when we move one element in the column indices (resp. the row indices) to the row indices (resp. the column indices).	

	 Note that the resulting decomposition is not necessarily canonical, as we may add a twin of a vertex in a prime graph.
	 But this is not a problem when we apply Lemma~\ref{lem:genupperbound}.
	 By the above modification, we know that for each $i\in \{1, 2, \ldots, n\}$, 
	 there is a linear layout of $B_i$ of width at most $p+2$ whose first and last vertices are $a_i$ and $b_i$, respectively.

	\begin{figure}
\centerline{
\includegraphics[scale=0.6]{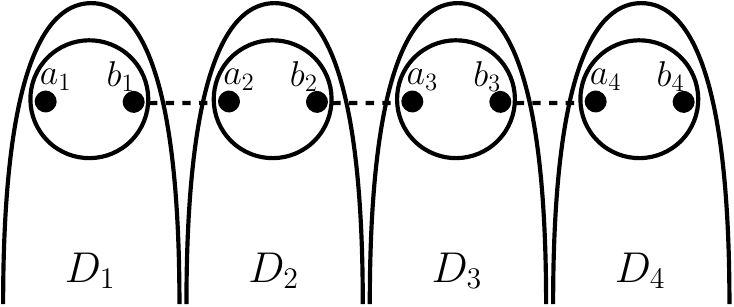}}
\caption{The sequence of sub-decompositions $D_1, \ldots, D_n$ in Proposition~\ref{prop:generalupperbound}. }
\label{fig:prop41}
\end{figure}

 We define the following sub-decompositions. See Figure~\ref{fig:prop41} for an illustration.
 If $n=1$, then let $D_1:=D$.
 Otherwise, 
 \begin{enumerate}
 \item let $D_1$ be the connected component of $D- V(B_2)$ containing  $B_1$,
 \item let $D_{n}$ be the connected component of $D- V(B_{n-1})$ containing $B_{n}$, and
 \item for each $i\in \{2, 3, \ldots, n-1\}$,
 let $D_i$ be the connected component of $D- (V(B_{i-1})\cup V(B_{i+1}))$ containing $B_i$. 
 \end{enumerate}
 We regard the vertices $a_i$ and $b_i$ as unmarked vertices of $D_i$.

Recall that $\pw(T)\le k-1$ for every node $v$ of $P$ and every connected component $T$ of $T_D- v$ not intersecting $P$. 
 Therefore, $\lrw(\origin{H})\le 2(p+2)k$, for each connected component $H$ of $D_i-V(B_i)$, by induction hypothesis.
   Thus, by Lemma~\ref{lem:genupperbound}, 
  $\origin{D_i}$ has a linear layout $L_i$ of width at most $2(p+2)+2(p+2)k=2(p+2)(k+1)$ 
  whose first and last vertices are $a_i$ and $b_i$, respectively.  For each $i\in \{1, 2, \ldots, n\}$, 
  let $L_i'$ be the linear layout obtained from $L_i$
 by removing $a_i$ and $b_i$.  
 Then it is not hard to check that
 \[ L'_1 \oplus L'_2\oplus \cdots \oplus L'_{n} \] is a linear layout of $G$ having width at most $2(p+2)(k+1)$.  We conclude that $\lrw(G) \le 2(p+2)(\pw(T_D)+1)$.
\end{proof}

For distance-hereditary graphs, the following establishes a lower bound and the tight upper bound of linear rank-width with respect to the path-width of their canonical split decompositions.

\begin{proposition}\label{prop:lrwpw}
  Let $D$ be the canonical split decomposition of a connected distance-hereditary graph $G$.
  Then $\frac{1}{2}\pw(T_D) \le \lrw(G) \le \pw(T_D)+1$.
\end{proposition}

The upper bound part is tight. For instance, every complete graph with at least two vertices has linear rank-width $1$ and the path-width of its decomposition tree has path-width $0$. Also, for each
odd integer $k=2n+1$ with $n\ge 1$, every complete binary tree of height $k$ (each path from a leaf to the root has distance $k$) has linear rank-width $\lceil k/2 \rceil=n+1$, and its decomposition
tree has path-width $\lceil (k-1)/2 \rceil=n$. (Note that the linear rank-width and the path-width of a tree are the same~\cite{AdlerK13}.) We will need the following lemmas.

\begin{lemma}\label{lem:decpw}
  Let $G$ be a graph and let $uv\in E(G)$.  Then $\pw(G)\le \pw(G/uv )+1$.
\end{lemma}

\begin{proof}
  Let $(P, \mathcal{B})$ be an optimal path-decomposition of $G/uv$, 
  and let $z$ be the contracted vertex in $G/uv$.  
  It is not hard to check that a new path-decomposition obtained by removing $z$ 
  and adding $u$ and $v$ in each bag
  containing $z$ is a path-decomposition of $G$.  
  We conclude that $\pw(G)\le \pw(G/uv)+1$.
\end{proof}

\begin{lemma}\label{lem:decpw2} Let $G$ be a graph. Let $u$ be a vertex of degree $2$ in $G$ such that $v_1, v_2$ are the neighbors of $u$ in $G$ and $v_1v_2\notin E(G)$.  Then $\pw(G)\le
\pw(G/uv_1/uv_2 )+1$.
\end{lemma}

\begin{proof} Let $w$ be the contracted vertex in $G/uv_1/uv_2$, and let $(P, \mathcal{B})$ be an optimal path-decomposition of $G/uv_1/uv_2$ of width $t:=\pw(G/uv_1/uv_2)$.  We may assume that no two adjacent bags in $(P, \mathcal{B})$ are equal.
  
  We obtain a path-decomposition $(P, \mathcal{B}')$ from $(P, \mathcal{B})$ by replacing $w$ with $v_1$ and $v_2$ in all bags containing $w$.  Since no two adjacent bags in $(P, \mathcal{B})$ are equal, no two adjacent bags in $(P, \mathcal{B}')$ are equal.
  
  We first assume that there are two adjacent bags $B_1$ and $B_2$ in $(P, \mathcal{B}')$ containing both $v_1$ and $v_2$, respectively.  We obtain a path-decomposition $(P', \mathcal{B}'')$ from $(P,
  \mathcal{B}')$ by subdividing the edge between $B_1$ and $B_2$, and adding a new bag $B'=(B_1\cap B_2) \cup \{u\}$.  Since $B_1$ and $B_2$ are not the same, $\abs{B_1\cap B_2}\le t+1$ and therefore,
  $\abs{B'}\le t+2$.  Thus, $(P', \mathcal{B}'')$ is a path-decomposition of $G$ of width at most $t+1$, and $\pw(G)\le \pw(G/uv_1/uv_2)+1$.
  
  Now we may assume that there is only one bag $B$ in $(P, \mathcal{B}')$ containing both $v_1$ and $v_2$.  In this case, since $v_1v_2\notin E(G)$, we can obtain a path decomposition of $G$ by replacing
  this bag $B$ with a sequence of two bags $B_1$ and $B_2$, where $B_1:=B\setminus \{v_2\} \cup \{u\}$ and $B_2:=B\setminus \{v_1\} \cup \{u\}$.  This implies that $\pw(G)\le \pw(G/uv_1/uv_2)+1$.
\end{proof}

We are now ready to prove Proposition \ref{prop:lrwpw}. We need the split decomposition characterization of graphs of linear rank-width at most $1$ proved by Bui-Xuan, Kant\'{e}, and Limouzy~\cite{Bui-XuanKL13} for the base case, which can be easily obtained by Theorem~\ref{thm:mainchap2}. We give a proof of this characterization in Theorem~\ref{thm:charlrw1}.

\begin{proof}[Proof of Proposition \ref{prop:lrwpw}] 
	  (1)~Let us first prove that $\pw(T_D)\leq 2\lrw(G)$ by induction on the linear rank-width of $G$.
	  Let $k:=\lrw(G)$.   
	  If $k=0$, then $G$ consists of a vertex, and $\pw(T_D)=0$.  
	  If $k=1$, then by Theorem~\ref{thm:charlrw1}, $T_D$ is a path and we have $\pw(T_D)\le 1\le 2k$.  
	  Thus, we may assume that $k\ge 2$.
	  By Theorem~\ref{thm:mainchap2}, there exists a path $P$ in $T_D$ such that 
	  \begin{itemize}
	  \item for every node $v$ in $P$ and every connected component $H$ of $D- V(\bag{D}{v})$ containing no bag in $\{\bag{D}{w}\mid w\in V(P)\}$, $f_D(\bag{D}{v}, H)\le k-1$. 
  \end{itemize}
	  Let $v$ be a node of $P$ and $C$ be a connected component of $D-V(\bag{D}{v})$ 
	  containing no bag $\bag{D}{w}$ with $w\in V(P)$. 
  	Let $y$ be an unmarked vertex of $C$ represented by $\zeta_c(D,\bag{D}{v},C)$, and let $L:=\limbtil_D[V(\bag{D}{v}),y]$. 
	By 	induction hypothesis, the decomposition tree $T_{L}$ of $L$ has path-width at most
  	$2k-2$.  %
  	We claim that $\pw(T_C)\le 2k-1$, where $T_C$ is the decomposition tree of $C$.  By the definition of canonical limbs, 
	either $T_L=T_C$ or $T_{L}$ is obtained from $T_C$ using one of the following operations:
  \begin{enumerate}
  \item Removing a node of degree $1$.
  \item Removing a node of degree $2$ with its neighbors $v_1, v_2$ and adding an edge $v_1v_2$.
  \item Removing a node of degree $2$ with its neighbors $v_1, v_2$ and identifying $v_1$ and $v_2$.
  \end{enumerate} 

  	The first two cases can be regarded as contracting one edge.  So, $\pw(T_C)\le \pw(T_{L})+1\le (2k-2)+1=2k-1$ 
	by Lemma~\ref{lem:decpw}.  The last case corresponds to contracting two edges incident with a vertex of degree $2$.  
  	By Lemma~\ref{lem:decpw2}, $\pw(T_C)\le \pw(T_{L})+1\le 2k-1$.

  	Therefore, for each node $v$ of $P$ and each connected component $T'$ of $T_D-v$ not containing a node of $P$ 
	we have that $\pw(T')\leq 2k-1$. By Theorem~\ref{thm:maintreeforpathwidth}, 
	$T_D$ has path-width at most $2k$, as required.

\medskip

	  (2)~We prove that $\lrw(G)\leq \pw(T_D)+1$ by induction on the path-width of $T_D$.
	  Let $k:=\pw(T_D)$   
	  If $k=0$, then $T_D$ consists of one node.
	  Since $G$ is distance-hereditary, $G$ should be a star or a complete graph, and therefore, 
	  we have $\lrw(G)\le 1=\pw(T_D)+1$.  
	  We may assume that $k\ge 1$.
	
  	 By Theorem~\ref{thm:maintreeforpathwidth}, 
	 there exists a path $P=v_0v_1 \cdots v_nv_{n+1}$ in $T_D$ such that for every node $v$ in $P$ and 
	 every connected component $F$ of $T_D- v$ containing no nodes of $P$, $\pw(F)\le k-1$.  
	Let $v$ be a node of $P$ and let $C$ be 
	a connected component of $D- V(\bag{D}{v})$ conaining no bags $\bag{D}{w}$ with  $w\in V(P)$.  
	By induction hypothesis, $\origin{C}$ has linear rank-width at most $(k-1)+1=k$. 
	By the definition of limbs, we conclude that $f_D(\bag{D}{v},C) \leq k$.  
	Thus, by Theorem~\ref{thm:mainchap2}, we conclude that $\lrw(G)\le k+1$.
\end{proof}

We could not confirm that the lower bound in Proposition~\ref{prop:lrwpw} is tight. We leave the following as an open question.  
\begin{QUE}
  Let $D$ be the canonical split decomposition of a connected distance-hereditary graph $G$.
  Is it true that $\pw(T_D) \le \lrw(G)$?
\end{QUE}

\section{Containing a tree as a vertex-minor}\label{sec:treevertexminor}

In this section, we prove our first main result.

\begin{THMMAIN}
  Let $p$ be a positive integer and let $T$ be a tree. Let $G$ be a graph such that every prime induced subgraph of $G$ has linear rank-width at most $p$.
  If $\lrw(G) \ge 40(p+2)\abs{V(T)}$, then $G$ contains a vertex-minor isomorphic to $T$.	
\end{THMMAIN}

	To prove it, we observe that the decomposition tree of the canonical split decomposition of $G$ has large path-width using Theorem~\ref{prop:generalupperbound}. 
	The main argument of this section is that  
	if $G$ admits a canonical split decomposition whose decomposition tree has sufficiently large path-width, 
	then $G$ contains a vertex-minor isomorphic to $T$.

	We first prove that every tree is a vertex-minor of some subcubic tree having slightly more vertices.
	For a tree $T$, we denote by $\phi(T)$ the sum of the degrees of vertices of $T$ whose degrees are
	at least $4$.  Every subcubic tree $T$ satisfies that $\phi(T)=0$.

\begin{figure}[t]\centering
\tikzstyle{v}=[circle, draw, solid, fill=black, inner sep=0pt, minimum width=2.5pt]
\tikzset{photon/.style={decorate, decoration={snake}}}
\begin{tikzpicture}[scale=0.4]

\node [v] (a) at (0,0) {};
\node [v] (b) at (2,3) {};
\node [v] (c) at (-2,3) {};
\node [v] (d) at (-4,0) {};
\node [v] (e) at (-2,-3) {};
\node [v] (f) at (2,-3) {};
\node [v] (g) at (4,0) {};

\draw (1,0.5) node{$v$};

\draw (4,0.7) node{$v_1$};
\draw (2,3.7) node{$v_2$};
\draw (-2,3.7) node{$v_3$};
\draw (-4,0.7) node{$v_4$};
\draw (-2.2,-2.3) node{$v_5$};
\draw (2.2,-2.3) node{$v_6$};

\draw (a)--(b);
\draw (a)--(c);
\draw (a)--(d);
\draw (a)--(e);
\draw (a)--(f);
\draw (a)--(g);

\end{tikzpicture}\qquad\quad
\begin{tikzpicture}[scale=0.4]

\node [v] (a) at (0,0) {};
\node [v] (b) at (10,3) {};
\node [v] (c) at (-2,3) {};
\node [v] (d) at (-4,0) {};
\node [v] (e) at (-2,-3) {};

\node [v] (f) at (2,-3) {};
\node [v] (g) at (12,0) {};

\node [v] (p1) at (4,0) {};
\node [v] (p2) at (8,0) {};

\draw (1,0.5) node{$v$};

\draw (12,0.7) node{$v_1$};
\draw (10,3.7) node{$v_2$};
\draw (-2,3.7) node{$v_3$};
\draw (-4,0.7) node{$v_4$};
\draw (-2.2,-2.3) node{$v_5$};
\draw (2.2,-2.3) node{$v_6$};

\draw (4,0.7) node{$p_2$};
\draw (7.8,0.7) node{$p_1$};

\draw (p2)--(b);
\draw (a)--(c);
\draw (a)--(d);
\draw (a)--(e);
\draw (a)--(f);
\draw (a)--(p1)--(p2)--(g);

\end{tikzpicture}
\caption{Splitting an edge in Lemma~\ref{lem:subcubicpivot1}.}\label{fig:splitting}
\end{figure}
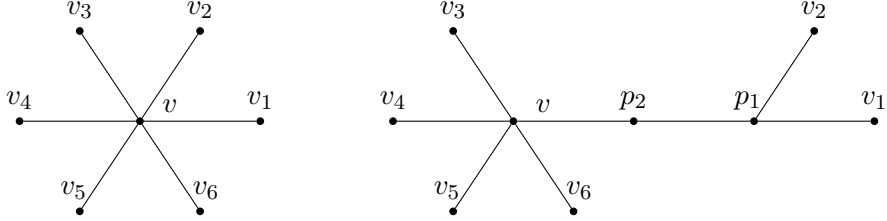

\begin{lemma}\label{lem:subcubicpivot1}
  Let $k$ be a positive integer and let $T$ be a tree with $\phi(T)=k$.  Then $T$ is a vertex-minor of a tree $T'$ with $\phi(T')=k-1$ and $\abs{V(T')}=\abs{V(T)}+2$.
\end{lemma}

\begin{proof}
	  Since $\phi(T)\ge 1$, $T$ has a vertex of degree at least $4$.  
	  Let $v\in V(T)$ be a vertex of degree at least $4$, and let $v_1, v_2, \ldots, v_m$ be its neighbors.  We obtain $T'$ from $T$ by
  	replacing the edge $vv_1$ with the path $vp_2p_1v_1$, removing $vv_2$ and adding an edge between $p_1$ and $v_2$.  
	It is easy to verify that $(T'\wedge p_1p_2)- \{p_1, p_2\}=T$.  
	We depict this procedure in Figure~\ref{fig:splitting}. 
	We observe that $p_1$ and $p_2$ are vertices of degree at most $3$ in $T'$, 
	and the degree of $v$ in $T'$ is one less than the degree of $v$ in $T$. 
	Therefore, we have $\phi(T')=k-1$.
\end{proof}	

\begin{lemma}\label{lem:subcubicpivot2}
  Every tree $T$ is a vertex-minor of a subcubic tree $T'$ with $\abs{V(T')}\le 5\abs{V(T)}$.
\end{lemma}

\begin{proof}
  By Lemma~\ref{lem:subcubicpivot1}, $T$ is a vertex-minor of a subcubic tree $T'$ with $\abs{V(T')}\le \abs{V(T)}+2\phi(T)$.  Since $\phi(T)\le 2\abs{E(T)}\le 2\abs{V(T)}$, we conclude that
  $\abs{V(T')}\le \abs{V(T)}+2\phi(T)\le 5\abs{V(T)}$.
\end{proof}

	We recall that by (2) of Theorem~\ref{thm:Bouchet88}, a connected
	graph is a tree if and only if every bag of its canonical split decomposition is a star bag whose center is an unmarked vertex.  
	The basic strategy is to extract the canonical split decomposition of a subcubic tree from the canonical split decomposition of $G$.
	To do this, we will obtain a star from each prime bag, without changing too much the shape of the obtained canonical split decomposition. 
	Lemma~\ref{lem:findingpathinprime} describes how to obtain a star from a prime graph as a vertex-minor, 
	without applying local complementations at some special vertices, which will correspond to marked vertices.	

	We observe that every prime graph on at least $5$ vertices is $2$-connected.
	This is because if a connected graph $G$ contains a cut vertex $v$ and $T_1, T_2, \ldots, T_m$ are connected components of $G-v$, 
	then $\left( V(T_1)\cup \{v\}, \bigcup_{j\in \{2, \ldots, m\}}V(T_j) \right)$ is a split of $G$.
	We use this observation in Lemma~\ref{lem:findingpathinprime}.

\begin{lemma}\label{lem:pathlength2}
Let $abc$ be an induced path in a $2$-connected graph $G$. 
By applying local complementations at vertices in $V(G)\setminus \{a,b\}$, 
we can obtain $G'$ locally equivalent to $G$ such that $G'[\{a,b,c\}]$ is a triangle.
\end{lemma}
\begin{proof}
As $b$ is not a cut vertex of $G$, 
there is a path from $a$ to $c$ in $G-b$.
Let $r_1r_2 \cdots r_s$ be the shortest path from $c=r_1$ to $a=r_s$ in $G-b$.
Note that $s\ge 3$ as $a$ is not adjacent to $c$. 
See Figure~\ref{fig:g2tog3} for an illustration.

\begin{figure}[t]\centering
\tikzstyle{v}=[circle, draw, solid, fill=black, inner sep=0pt, minimum width=2.5pt]
\tikzset{photon/.style={decorate, decoration={snake}}}
\begin{tikzpicture}[scale=0.4]

\node [v] (a) at (0,3) {};
\node [v] (c) at (-4,0) {};
\node [v] (d) at (-2,-3) {};
\node [v] (e) at (0,-3.5) {};
\node [v] (f) at (2,-3) {}; 
\node [v] (g) at (4,0) {};

\draw (0,3.7) node{$b$};
\draw (-4.5,-0.7) node{$r_1$};
\draw (-2,-3.7) node{$r_2$};
\draw (0,-4.2) node{$r_3$};
\draw (2,-3.7) node{$r_4$};
\draw (4.5,-0.7) node{$r_5$};

\draw (a)--(c)--(d)--(e)--(f)--(g)--(a);
\draw (a)--(d);
\draw (a)--(f);
\draw (a)--(g);

\draw (0,-7) node{$G[\{b, r_1, r_2, \ldots, r_5\}]$};

\end{tikzpicture}\qquad\quad
\begin{tikzpicture}[scale=0.4]

\node [v] (a) at (0,3) {};
\node [v] (c) at (-4,0) {};
\node [v] (d) at (-2,-3) {};
\node [v] (e) at (0,-3.5) {};
\node [v] (f) at (2,-3) {}; 
\node [v] (g) at (4,0) {};

\draw (0,3.7) node{$b$};
\draw (-4.5,-0.7) node{$r_1$};
\draw (-2,-3.7) node{$r_2$};
\draw (0,-4.2) node{$r_3$};
\draw (2,-3.7) node{$r_4$};
\draw (4.5,-0.7) node{$r_5$};

\draw (a)--(c)--(d)--(e)--(f)--(g)--(a);
\draw[dashed] (a)--(d);
\draw (a)--(f);
\draw (a)--(g);
\draw (c)--(e);

\draw (0,-7) node{$G[\{b, r_1, r_2, \ldots, r_5\}]*r_1*r_2$};

\end{tikzpicture}\caption{Reducing from $G[\{b, r_1, r_2, \ldots, r_s\}]$ in Lemma~\ref{lem:pathlength2}.}\label{fig:g2tog3}
\end{figure}
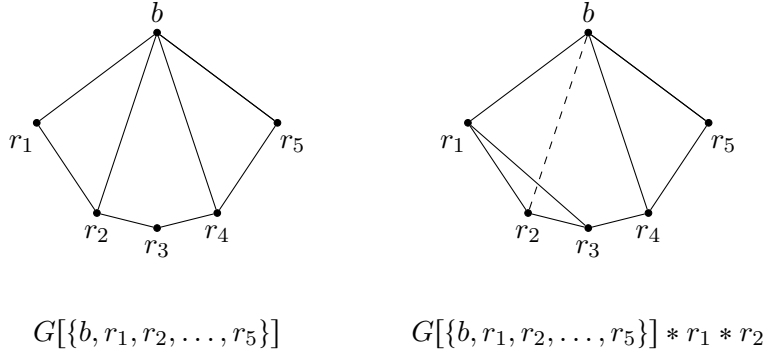

We prove by induction on $s$ that 
$G[\{b, r_1, r_2, \ldots, r_s\}]$ can be transformed into an induced path $acb$ by applying local complementations only at vertices in $\{r_1, r_2, \ldots, r_{s-1}\}$.
We illustrate this procedure in Figure~\ref{fig:g2tog3}.
Assume $s=3$. If $b$ is adjacent to $r_2$, then we remove this edge by applying a local complementation at $c=r_1$. And then 
we apply a local complementation at $r_2$ to create an edge between $a$ and $c$.
Then $abc$ becomes a triangle.

We assume $s\ge 4$. Similarly, 
if $b$ is adjacent to $r_2$, then we remove this edge by applying a local complementation at $c=r_1$, 
and then we apply a local complementation at $r_2$ to create an edge between $c$ and $r_3$.
If $b$ is not adjacent to $r_2$, then we apply a local complementation at $r_2$ to create an edge between $c$ and $r_3$.
Let $G_1$ be the resulting graph.
Then $r_1r_3r_4 \cdots r_s$ is an induced path in $G_1-b$.
Thus, by induction hypothesis, 
we can obtain $G_2$ locally equivalent to $G_1[\{b, r_1, r_3, r_4, \ldots, r_s\}]$ by applying local complementations only at vertices in $\{r_1, r_3, \ldots, r_{s-1}\}$
such that $G_2[\{a,b,c\}]$ is a triangle.
\end{proof}

\begin{lemma}\label{lem:findingpathinprime}
 Let $G$ be a prime graph on at least $5$ vertices, and let $a,b,c\in V(G)$.  
 Then there exists a sequence $x_1, x_2, \ldots, x_t$ of vertices in $V(G)\setminus \{a, b\}$ (not necessarily all distinct) such that $acb$ is an induced path of $G*x_1*x_2*\cdots *x_t$. 
\end{lemma}
\begin{proof}

	We first create a triangle or an induced path of length $2$ on $\{a,b,c\}$ by applying local complementations at vertices in $V(G)\setminus \{a,b,c\}$.
	For this argument, $a,b,c$ are symmetric.
	Without loss of generality, we assume the distance between $a$ and $b$ is at most the distance between $a$ and $c$ or between $b$ and $c$.
	Let $P=p_1p_2 \cdots p_m$ be a shortest path from $a=p_1$ to $b=p_m$ in $G$.
	By the distance property, $c\notin V(P)$.
	We define \[
    G_1:=
    \begin{cases}
      G*p_2*p_3* \cdots *p_{m-1}&\text{if $m\ge 3$},\\
      G &\text{otherwise}.
    \end{cases}
  \]
  It is not difficult to observe that $a$ and $b$ are adjacent in $G_1$.
Now, we take a shortest path $Q=q_1q_2 \cdots q_n$ from $c=q_1$ to $q_n\in \{a,b\}$ in $G_1$.
We define \[
    G_2:=
    \begin{cases}
      G_1*q_2*q_3*\cdots *q_{n-1}&\text{if $n\ge 3$},\\
      G_1 &\text{otherwise}.
    \end{cases}
  \]
We observe that 
$c$ has a neighbor on $\{a,b\}$ in $G_2$.
Furthermore, if $a$ and $b$ are not adjacent in $G_2$, it means that 
the last local complementation removed this edge, and it implies that $c$ should be adjacent to both $a$ and $b$ in $G_2$.
Therefore, either $G_2[\{a,b,c\}]$ is a triangle or an induced path of length $2$.

We do not want to apply local complementation at $a, b$ to create a required induced path.
If $acb$ is already an induced path, then we are done.
If $G_2[\{a,b,c\}]$ is a triangle, then we apply local complementation at $c$.
Therefore, we may assume that $abc$ or $bac$ is an induced path.
Note that $G_2$ is $2$-connected.

  \vskip 0.2cm
  \noindent\emph{\textbf{Case 1.} $abc$ is an induced path in $G_2$.}

We apply Lemma~\ref{lem:pathlength2}.
Then by applying local complementations at vertices in $V(G)\setminus \{a,b\}$, 
we can obtain $G_3$ locally equivalent to $G_2$ such that $G_3[\{a,b,c\}]$ is a triangle.
By applying a local complementation at $c$, we obtain the required path.

  \vskip 0.2cm
  \noindent\emph{\textbf{Case 2.} $bac$ is an induced path in $G_2$.}

We apply Lemma~\ref{lem:pathlength2}.
Then by applying local complementations at vertices in $V(G)\setminus \{a,b\}$, 
we can obtain $G_3$ locally equivalent to $G_2$ such that $G_3[\{a,b,c\}]$ is a triangle.
By applying a local complementation at $c$, we obtain the required path.

\vskip 0.2cm

We conclude the lemma.
\end{proof}

Starting from a split decomposition whose decomposition tree is a subdivision of a huge binary tree, 
we will extract a split decomposition of some fixed binary tree. To do this, we need to explain how we sequentially transform each bag into a star whose center is unmarked.
Lemma~\ref{lem:primetwomarked1} deal with the case when a bag has two neighbor bags, 
and Lemma~\ref{lem:primetwomarked2} deal with the case when a bag has three neighbor bags.

A canonical split decomposition $D$ is \emph{rooted} if we distinguish a leaf bag and call it the \emph{root} of $D$. Let $D$ be a rooted canonical split decomposition with root bag $R$. A bag $B$ is a \emph{descendant}
of a bag $B'$ if $B'$ is on the path of bags from $R$ to $B$ in $D$. 
If $B$ is a descendant of $B'$ and $B$ and $B'$ are adjacent bags, then we call $B$ a \emph{child} of
$B'$ and $B'$ the \emph{parent} of $B$. A bag in $D$ is called a \emph{non-root bag} if it is not the root bag.

	\begin{figure}
\centerline{
\includegraphics[scale=0.5]{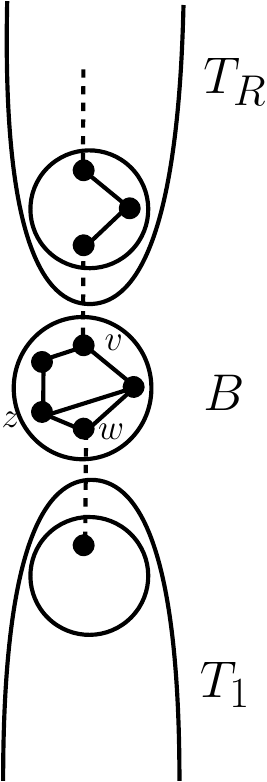}
\qquad\qquad\qquad
\includegraphics[scale=0.5]{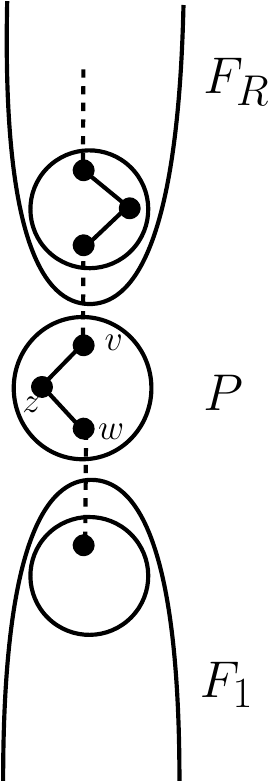}}
\caption{An example application of Lemma~\ref{lem:primetwomarked1}. }
\label{fig:twobags}
\end{figure}

\begin{lemma}\label{lem:primetwomarked1}
  Let $D$ be a rooted canonical split decomposition of a connected graph with root bag $R$ and 
  let $B$ be a non-root bag of $D$
 	such that
   	\begin{itemize}
   	\item $D- V(B)$ has
  exactly two connected components $T_1$ and $T_R$ where $T_R$ contains $R$, 
  \item the parent of $B$ is a star and $\zeta_c(D,B,T_R)$ is a leaf.  
  \end{itemize}
  Then by possibly applying local
  complementations at unmarked vertices of $D$ contained in $V(T_1)\cup V(B)$ and 
    deleting some unmarked vertices in $B$, we can transform $D$ into a canonical split decomposition $D'$ containing a bag $P$ such that 
  \begin{enumerate}
  \item $D'- V(P)$ consists of exactly two connected components $F_R$ and $F_1$,
  \item $F_R=T_R$ or $F_R=T_R*\zeta_c(D,B,T_R)$, 
  \item $F_1$ is locally equivalent to $T_1$, and
  \item $P$ is a star bag whose center is unmarked.
  \end{enumerate}
\end{lemma}
\begin{proof}
  Let $v:=\zeta_b(D,B,T_R)$ and $w:=\zeta_b(D,B, T_1)$.
  Let $y$ be an unmarked vertex in $D$ represented by $w$.
  See Figure~\ref{fig:twobags} for the setting.

  First assume that $B$ is a star bag. Since $\zeta_c(D,B,T_R)$ is a leaf, $v$ is not the center of $B$.  If its center is unmarked, then we are done.  We may assume the center of $B$
  is $w$.  Since $\abs{V(B)}\ge 3$, $B$ contains at least one unmarked vertex, which is adjacent to $w$.  We choose an unmarked leaf vertex $z$ in $B$.
  We observe that $y$ is linked to $z$, that is, $yz\in E(G)$.
  Then in $D\wedge yz$, $z$ becomes the center of a star, and $T_R$ does not change.  Also, $T_1$ is changed to the decomposition obtained from $T_1$ by pivoting $yz'$ where $z'=\zeta_c(D, B, T_1)$.
   Thus, the resulting decomposition satisfies the required property.  If $B$ is a complete bag, then we choose an unmarked vertex in $B$, and apply a local
  complementation at this vertex.  Then the resulting decomposition satisfies the required property.

  Now, suppose $B$ is a prime bag. 
 Choose an unmarked vertex $z$ of $B$ that is adjacent to $w$.
 Since a prime graph with at least $5$ vertices is $2$-connected, there is always an unmarked vertex adjacent to $w$.
 Note that $y$ and $z$ are linked.
 
  Let $B_1$ be the child of $B$. If $B_1$ is a star bag whose center is adjacent to $B$, 
  then by pivoting $yz$
  we transform $B_1$ into a star bag having $\zeta_c(D,B,T_1)$ as a leaf.
  If $B_1$ is a complete bag, then we apply a local complementation at $y$.
  In the resulting decomposition, either $B_1$ is a prime bag or $\zeta_c(D, B, T_1)$ is a leaf of a star bag.
  Let $B'$ be the bag modified from $B$ in the resulting decomposition. Note that $B'$ is still a prime graph by Lemma~\ref{lem:vm-rw}.
  
  We apply Lemma~\ref{lem:findingpathinprime} with $(a,b,c)=(v,w,z)$. 
 By Lemma~\ref{lem:findingpathinprime},
  we can modify $B'$ into an induced path $vzw$ by only applying local complementations at unmarked vertices in $B'$ and removing all unmarked vertices in $B'$ except $z$.  
    Note that the marked edges incident with $B'$ are still marked edges that cannot be recomposed, 
    as both have types $S_pS_p$ or $S_pP$. 
    Let $D'$ be the modified decomposition and let $P$ be the new bag in $D'$ modified from $B'$. 
    Then $D'- V(P)$ has two connected components $F_R$ and $F_1$ where 
    \begin{itemize}
    \item $F_R=T_R$  or $F_R=T_R*\zeta_c(D,B,T_R)$,
    \item $F_1$ is locally equivalent to $T_1$,  and
    \item $P$ is a star whose center is unmarked,
    \end{itemize} as required.
\end{proof}

\begin{lemma}\label{lem:primetwomarked2}
  Let $D$ be a rooted canonical split decomposition of a connected graph with root bag $R$ and let $B$ be a non-root bag of $D$ such that
  \begin{itemize}
  \item $D- V(B)$ has exactly three connected components $T_1,T_2,$ and $T_R$ where $T_R$ contains $R$,
  \item the distance from $\node{D}{B}$ to $\node{D}{R}$ is at least $3$ in $T_D$,
  \item the parent $P_1$ of $B$ and its parent $P_2$ satisfy that 
  $\node{D}{P_1}$ and $\node{D}{P_2}$ have degree $2$ in $T_D$,
  \item $P_1$ and $P_2$ are stars whose centers are unmarked, and
  \item for each $i\in \{1,2\}$, the child $B_i$ of $B$ in $T_i$ satisfies that $\node{D}{B_i}$ has degree $2$ in $T_D$. 
  \end{itemize}
  Then by possibly applying local complementations at unmarked vertices of $D$ contained in $V(T_1)\cup V(T_2)\cup V(B)\cup V(P_1)\cup V(P_2)$ and
  deleting some unmarked vertices in $V(T_1)\cup V(T_2)\cup V(B)\cup V(P_1)\cup V(P_2)$ and recomposing some marked edges, we can transform $D$ into a canonical split decomposition $D'$ containing a bag $P$ such that 
  \begin{enumerate}
  \item $D'- V(P)$ consists of exactly three connected components $F_1, F_2,$ and $F_R$, 
  \item  $F_R=T_R- (V(P_1)\cup V(P_2) )$, 
  \item for each $i\in \{1,2\}$, $F_i$ is locally equivalent to $T_i$ or $T_i- V(B_i)$, and
  \item $P$ is a star bag whose center is unmarked.
  \end{enumerate}
\end{lemma}
\begin{proof}
  For each $i\in \{1,2\}$, let $x_i$ be the center of $P_i$, 
  and let $v:=\zeta_b(D,B,T_R)$, and for each $i\in \{1,2\}$, let $v_i:=\zeta_b(D, B, T_i)$, 
  and $y_i$ be an unmarked vertex represented by $v_i$.

	We first deal with an easier case.

\begin{figure}
\begin{center}
\begin{subfigure}[b]{0.23\textwidth}
\includegraphics[scale=0.33]{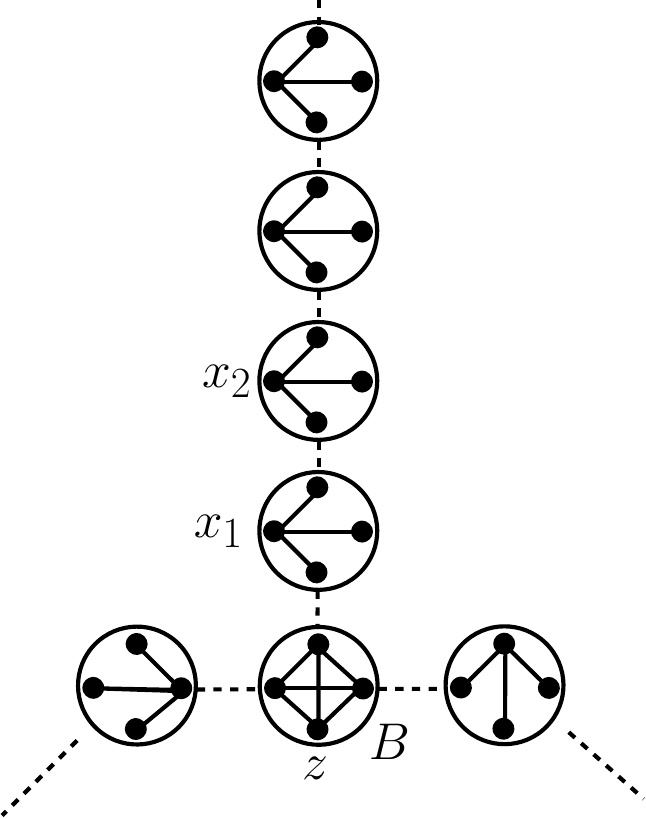} 
\caption{$D$}
\end{subfigure}
\qquad\qquad
\begin{subfigure}[b]{0.23\textwidth}
\includegraphics[scale=0.33]{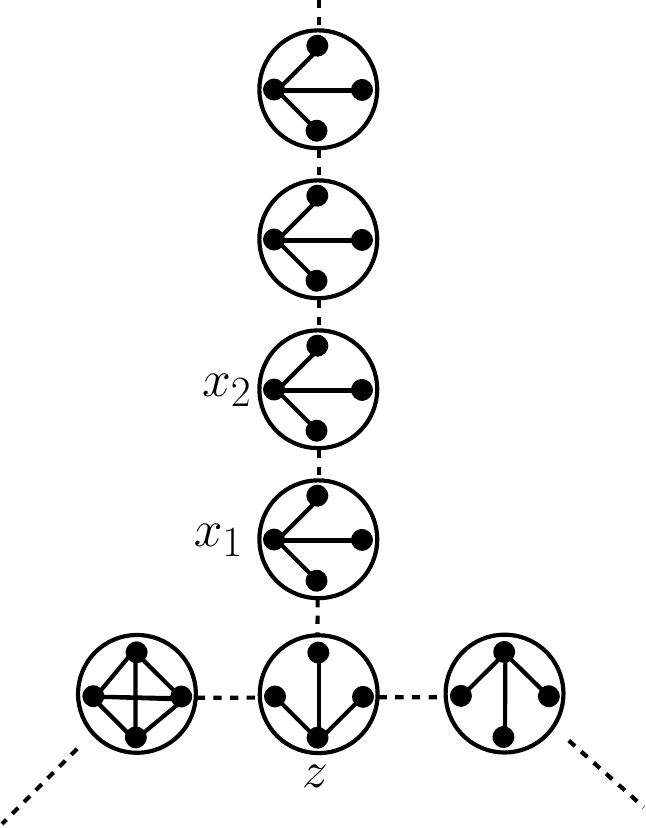} 
\caption{$D*z$}
\end{subfigure}
\end{center}
\vskip 0.3cm
\begin{center}
\begin{subfigure}[b]{0.23\textwidth}
\includegraphics[scale=0.33]{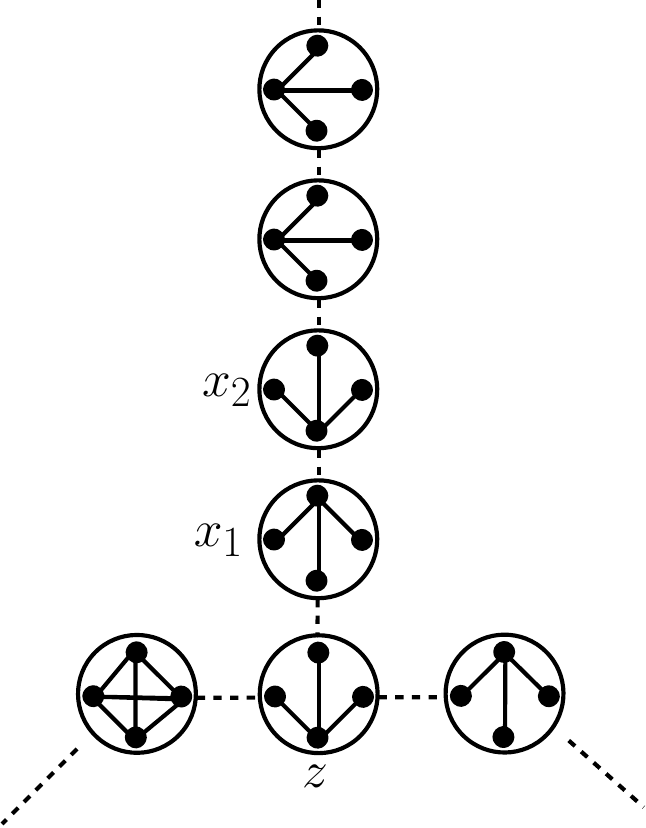}
\caption{$D*z\wedge x_1x_2$}
\end{subfigure}
\qquad\qquad
\begin{subfigure}[b]{0.23\textwidth}
\includegraphics[scale=0.33]{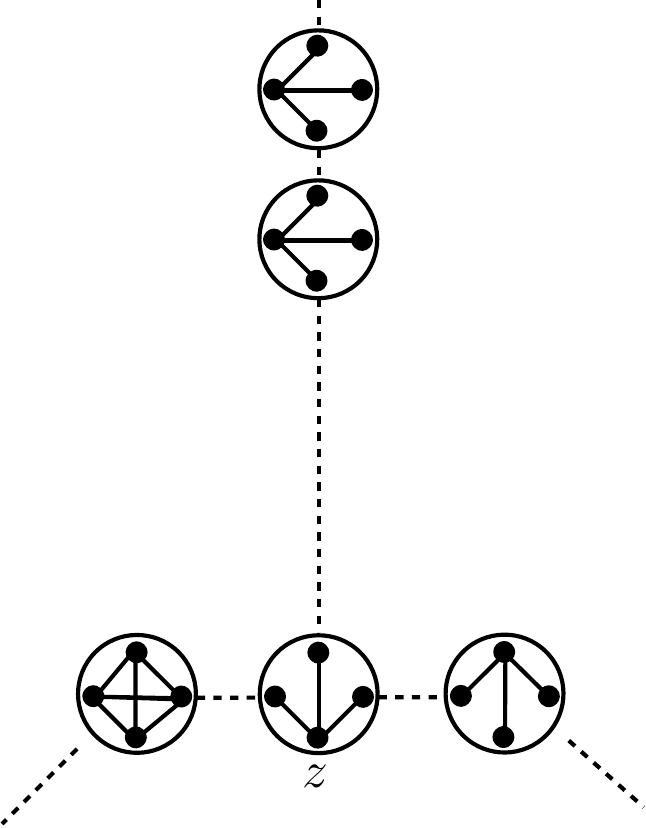}
\caption{$D*z\wedge x_1x_2-\{x_1, x_2\}$}
\end{subfigure}
\end{center}
\caption{When $B$ is a complete bag and has an unmarked vertex in Lemma~\ref{lem:primetwomarked2}. }
\label{fig:completebag}
\end{figure}

  \vskip 0.2cm
  \noindent\emph{\textbf{Case 1.} $B$ is either a star or a complete graph, and has an unmarked vertex.}

The case when $B$ is a complete graph is depicted in Figure~\ref{fig:completebag}.
	We first transform $B$ into a star whose center is unmarked.
	Let $z$ be an unmarked vertex in $B$.

  Assume $B$ is a star. Since $\zeta_c(D, B, T_R)$ is a leaf of a star,
	$v$ is not the center of $B$. 
	We may assume that the center of $B$ is either $v_1$ or $v_2$.
	By symmetry, we may assume it is $v_1$. 
	In this case, $y_1$ and $z$ are linked in $D$.
	Thus, $B$ becomes a star whose center is $z$ in $D\wedge y_1z$.
	If $B$ is a complete bag, then we apply a local complementation at $z$. 
	Then $B$ becomes a star whose center is $z$.
	Note that in any case, $T_R$ does not change by this local complementation as $\zeta_c(D, B, T_R)$ is a leaf of a star, 
	and $T_i$ becomes a split decomposition locally equivalent to $T_i$.
	
	Let $D_1$ be the resulting decomposition. Lastly,  we transform $D_1$ into a split decomposition $D_2$ as follows:
  \begin{enumerate}
  \item We pivot $x_1x_2$ and then remove all unmarked vertices contained in $P_1$ and $P_2$.
  \item We recompose marked edges incident with $P_1$ and $P_2$. Equivalently, we remove all vertices in $P_1$ and $P_2$ in the decomposition, and 
  add a new marked edge between $v$ and the marked vertex in the parent of $P_2$ that is adjacent to $P_2$.
  \end{enumerate}
  Note that $D_2$ is canonical, as the new marked edge has the same type as before.
  Thus, we obtained a required decomposition.

  \medskip  
  Now, we may assume that either $B$ is a prime bag, or $\abs{V(B)}=3$.  
    
    \begin{figure}
\begin{center}
\begin{subfigure}[b]{0.23\textwidth}
\includegraphics[scale=0.33]{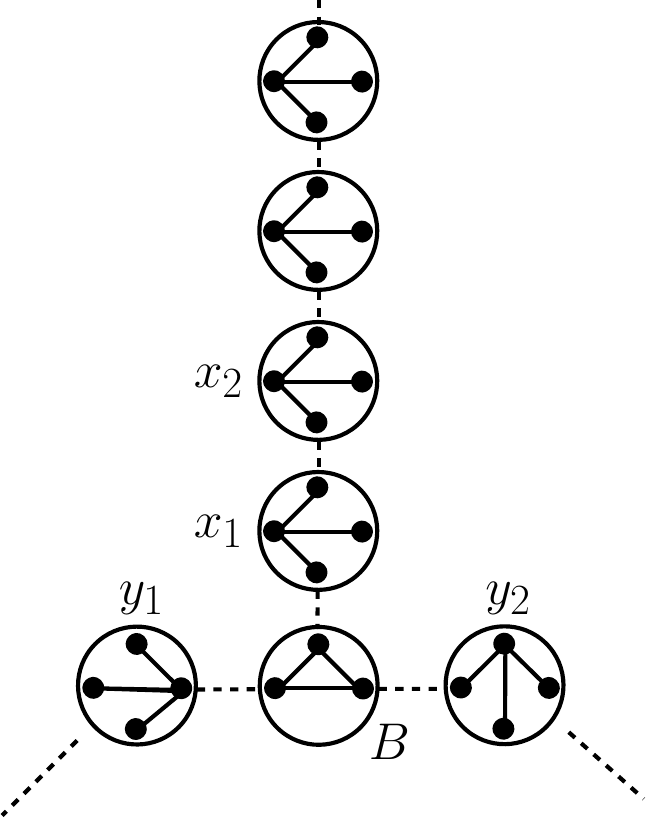} 
\caption{$D$}
\end{subfigure}
\qquad\qquad
\begin{subfigure}[b]{0.23\textwidth}
\includegraphics[scale=0.33]{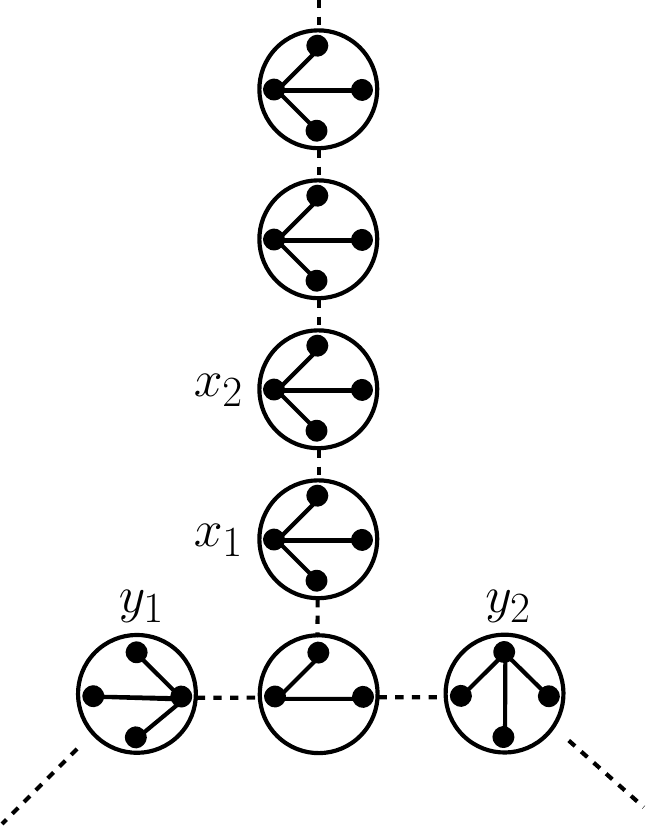} 
\caption{$D*y_1$}
\end{subfigure}
\end{center}
\vskip 0.3cm
\begin{center}
\begin{subfigure}[b]{0.23\textwidth}
\includegraphics[scale=0.33]{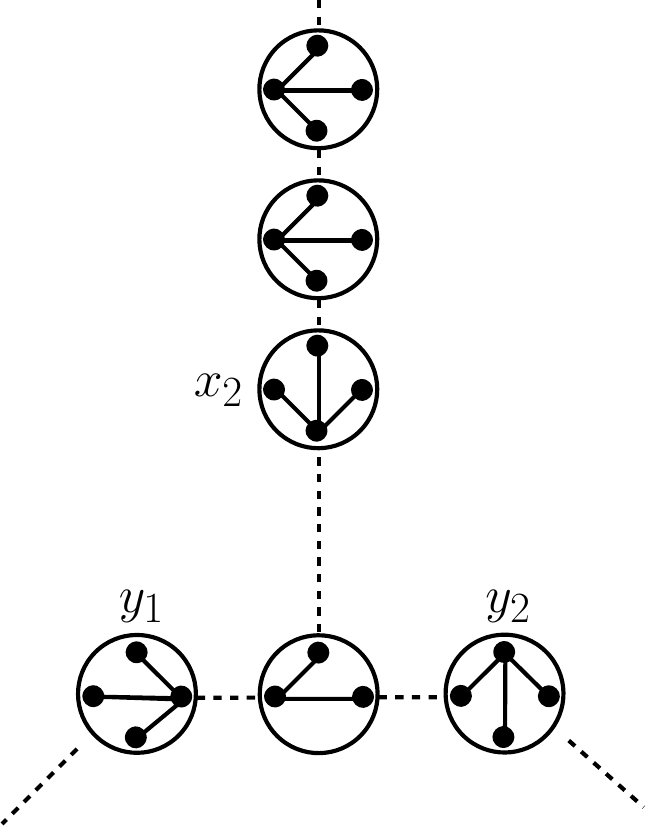}
\caption{$D*y_1\wedge x_1x_2-x_1$}
\end{subfigure}
\qquad\qquad
\begin{subfigure}[b]{0.23\textwidth}
\includegraphics[scale=0.33]{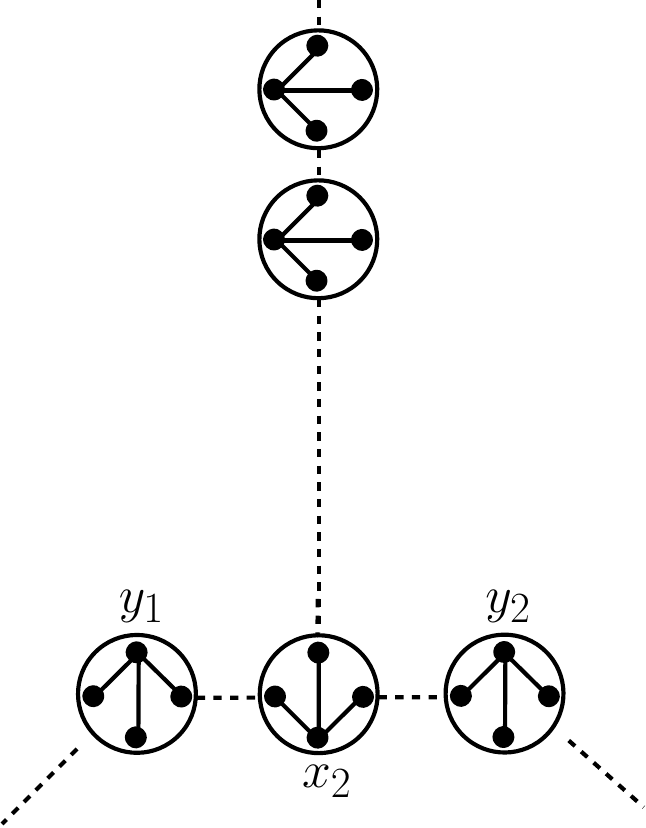}
\caption{Pivot $x_2y_1$}
\end{subfigure}
\end{center}
\caption{When $B$ is a complete bag and has no unmarked vertices in Lemma~\ref{lem:primetwomarked2}. }
\label{fig:completebagthreevertices}
\end{figure}

  \vskip 0.2cm
  \noindent\emph{\textbf{Case 2.} $\abs{V(B)}=3$.}

	An example case is depicted in Figure~\ref{fig:completebagthreevertices}.
	
	Since $\abs{V(B)}=3$, $B$ is either a star or a complete graph. We first modify $B$ into a star whose center is $v_1$.
	First assume that $B$ is a star. Since $\zeta_c(D, B, T_R)$ is a leaf of a star,
	the center of $B$ is either $v_1$ or $v_2$. We may assume the center of $B$ is $v_2$. 
	Since $v_1$ is adjacent to $v_2$, $y_1$ and $y_2$ are linked in $D$.
	Then $B$ becomes a star whose center is $v_1$ in $D\wedge y_1y_2$.
	If $B$ is a complete bag, then we apply local complementation at $y_1$. Then $B$ becomes a star whose center is $v_1$.
	Note that $T_R$ does not change by this local complementation as $\zeta_c(D, B, T_R)$ is a leaf of a star and the center of the parent of $B$ is unmarked. 
	Let $D_1$ be the resulting decomposition.
	
	Let $w$ be the marked vertex in $P_2$ that is adjacent to $P_1$. 
   We transform $D_1$ into a split decomposition $D_2$ as follows:
  \begin{enumerate}
  \item We pivot $x_1x_2$. 
  \item We delete the vertices of $V(P_1)$, and add a marked edge between $v$ and $w$.
  \item We recompose the new marked edge $vw$ (it is of type $S_pS_c$). 	
  \end{enumerate}
  Observe that the bag $B'$ in $D_2$ obtained by merging $B$ and $P_2$
  is a star whose center is $v_1$, and it contains an unmarked vertex $x_2$.
  Moreover, $D_2$ is canonical.
  Lastly, we pivot $y_1x_2$. Then $B'$ becomes a star whose center is $x_2$. 
  Note that the connected components of $D_2- V(B')$ are respectively $T_R- (V(P_1)\cup V(P_2) )$ and $F_1$ and $F_2$ such that
  $F_i$ is locally equivalent to $T_i$ for $i\in \{1,2\}$.
  
  \vskip 0.2cm
  
  Now, it remains to show when $B$ is a prime bag. 
  We reduce this case to \emph{\textbf{Case 1}} or  \emph{\textbf{Case 2}} applying Lemma~\ref{lem:findingpathinprime}.
  Note that in the previous cases, we deduce that $F_i$ is locally equivalent to $T_i$ for each $i\in \{1,2\}$.
  But when we transform $B$ into a star bag, we may merge $B$  with one of its child bags.
  \vskip 0.2cm
  \noindent\emph{\textbf{Case 3.} $B$ is a prime bag.}

	Note that applying a local complementation at an unmarked vertex in $B$ does not change the fact that $y_1$ is represented by $v_1$.
	This is because the alternating path from $y_1$ to $v_1$ does not change when we apply a local complementation at an unmarked vertex in $B$.
	
  	We apply Lemma~\ref{lem:findingpathinprime} with $(a,b,c)=(v, v_2, v_1)$ so that $B$ is transformed into an indued path $vv_1v_2$.
	Note that applying a local complementation at $v_1$ can be simulated by applying a local complementation at $y_1$.
    Since $B$ is a prime graph on at least $5$ vertices, 
    by Lemma~\ref{lem:findingpathinprime}, we can modify $B$ into an induced path $vv_1v_2$ by only applying local complementations at unmarked vertices in $B$ and $y_1$.
    Then we remove all the other vertices of $B$.  
    
    Note that the marked edge connecting $B$ and $P_1$ is still a valid marked edge as 
    $\zeta_c(D,B,T_R)$ is a leaf of a star.  However, for $i\in \{1,2\}$, the marked edge incident with $v_i$ and $\zeta_c(D,B,T_i)$ may have type $S_pS_c$. 
    In this case, we recompose this marked edge so that the resulting decomposition is canonical.
	
	Let $D_1$ be the modified decomposition.
  Since both $\node{D}{P_1}$ and $\node{D}{P_2}$ have degree $2$ in $T_D$, 
  the bag $B'$ of $D_1$ modified from $B$ still has $3$ adjacent bags in $D_1$.
    As $B'$ is star bag of $D_1$, we can reduce the remaining steps to \emph{\textbf{Case 1}} or \emph{\textbf{Case 2}} depending on the size of $B'$, 
    from which we can construct the required  canonical split decomposition.
\end{proof}

	We are ready to prove the main result of the section.
	We note that for a graph $H$, any subdivision of $H$ contains a vertex-minor isomorphic to $H$.
	We will use this fact.
	For a tree $T$, let $\eta(T)$ be the tree obtained from $T$ by replacing each edge with a path of length $4$.

\begin{proof}[Proof of Theorem~\ref{thm:largelrw}]
  	Let $t:=\abs{V(T)}$ and suppose that $\lrw(G) \ge 40(p+2)t$.
  	By Lemma~\ref{lem:subcubicpivot2}, 
  	there exists a subcubic tree $T'$ such that $T$ is a vertex-minor of $T'$ and $\abs{V(T')}\le 5t$. 
	We consider the tree $\eta(T')$ which is the tree obtained from $T'$ by replacing each edge with a path of length $4$.
	Observe that $\abs{V(\eta(T'))}\le 20t$.
  
  	Let $D$ be the canonical split decomposition of $G$ and let $T_D$ be the decomposition tree of $D$.  
	Since $\lrw(G) \ge 40(p+2)t$, by Proposition~\ref{prop:generalupperbound}, $\pw(T_D)\ge 20t-1$.  
	Since $\abs{V(\eta(T'))}\le 20t$, from Theorem~\ref{thm:pathwidththeorem}, 
	$T_D$ contains a minor isomorphic to $\eta(T')$.  
	Since the maximum degree of $\eta(T')$ is at most $3$, $T_D$ contains a subgraph $T_1$ that is isomorphic to a
  	subdivision of $\eta(T')$.  
	Let $D_1:=D[\bigcup_{v\in V(T_1)} V(\bag{D}{v})]$. 
	Observe that $D_1$ is not necessarily a decomposition of an induced subgraph of $G$, 
	as the unmarked vertex which was a marked vertex before does not correspond to a real vertex of $G$.
	To make it as a decomposition of an induced subgraph of $G$,
	we obtain a new decomposition $D_2$ from $D_1$ as follows:
	\begin{itemize}
	\item For every unmarked vertex $x$ of $D_1$ that was a marked vertex in $D$,  
	there is a vertex $y\in V(G)$ represented by $x$ in $D$.
	We choose such a vertex and replace $x$ with $y$.
	\end{itemize}
	We can observe that $D_2$ is a canonical split decomposition of an induced subgraph of $G$, 
	and $T_{D_2}$ is isomorphic to $T_{D_1}$.

	We choose a leaf bag $R_2$ of $D_2$ and regard it as the root of $D_2$. 
	We first transform $R_2$ into a star where the marked vertex in $R_2$ is a leaf by applying local complementations at unmarked vertices of $D_2$. 
	\begin{itemize}
	\item[($\ast$)] Let $v$ be the marked vertex of $R_2$, and $v'$ be a neighbor of $v$ in $R_2$, and $w$ be an unmarked vertex of $D_2$ represented by $v$.
	If $R_2$ is a star whose center is unmarked, then we do nothing. If $R_2$ is a star whose center is $v$, 
	then we pivot $v'w$. If $R_2$ is a complete bag, then 
	we apply local complementation at $v'$.
	Then $R_2$ becomes a star whose center is unmarked.

	Assume $R_2$ is a prime bag and let $C$ be the child of $R_2$. 
	If $C$ is a star whose center $c$ is adjacent to $v$, then we do a pivot at $v'w$ to turn $C$ into a star with $c$ as a leaf.
	If $C$ is a complete graph, then we apply a local complementation at $w$.
	The bag modified from $C$ is either a prime graph or a star whose leaf is adjacent to $v$.
	Let $R_2'$ be the resulting bag from $R_2$.
	
	Now, we choose one more unmarked vertex $v''$ in $R_2'$ adjacent to $v$. Such a vertex exists as $R_2'$ is $2$-connected.
	Applying Lemma~\ref{lem:findingpathinprime} to $R_2'$ with $(a,b,c)=(v, v', v'')$, 
	there exists a sequence $x_1, x_2, \ldots,x_{\ell}$ of vertices in $V(R_2')\setminus \{v, v'\}$ such that $vv''v'$ is an induced path of $R_2'*x_1*x_2*\cdots *x_{\ell}$.	We apply this sequence of local complementations and then remove all vertices in $R_2'$ except $v, v'$, and $v''$.
	By the previous procedure, the resulting decomposition is canonical and the bag modified from $R_2'$ is a star whose center is unmarked.	
	\end{itemize}
	Let $D_3$ be the resulting decomposition, and $R_3$ be the root bag that is modified from $R_2$.
	Note that $T_{D_3}$ is isomorphic to $T_{D_2}$.
	
	As $T_{D_3}$ is isomorphic to a subdivision of $\eta(T')$,
	there is a subdivision mapping $g$ from $T'$ to $T_{D_3}$ such that
	for each edge $e$ of $T'$, $g(e)$ is a path of length at least $4$. 
	Note that $g(V(T'))$ is exactly the set of all leaves and all vertices of degree at least $3$ in $T_{D_3}$.
	
	A bag $B$ is \emph{processed} if 
	every bag on the path from $B$ to the root bag is a star whose center is unmarked.
	Let $B_1, B_2, \ldots, B_m$ be an ordering of bags in $\{\bag{D_3}{v}:v\in g(V(T'))\}$ such that
	\begin{itemize}
	\item for each $i\in \{2, 3, \ldots, m\}$, every ascendant bag of $B_i$ in the set $\{\bag{D_3}{v}:v\in g(V(T'))\}$  
	is contained in $\{B_1, B_2, \ldots, B_{i-1}\}$.
	\end{itemize}
	Such an ordering can be found using a BFS.
	For each $i\in \{2,3, \ldots, m\}$, let $F(B_i)$ be the bag $B$ in $\{B_1, B_2, \ldots, B_{i-1}\}$ such that $B$ is an ascendant bag of $B_i$, 
	and $B$ is closest to $B_i$. We will define below a sequence $F_1,F_2,\ldots,F_m$ of rooted canonical split decompositions such that $\node{D_3}{B_j}\in V(T_{F_i})$ for $1\leq
        i,j\leq m$, and for convenience we keep continuing calling $B_j$ the bag $\bag{F_i}{\node{D_3}{B_j}}$. 

		For $j\in \{1, 2, \ldots, m\}$, let $F_1, F_2, \ldots, F_j$ be a maximal sequence of rooted canonical split decompositions such that
	\begin{itemize}
	\item $D_3=F_1$, 
	\item for each $i\in \{1, 2, \ldots, j-1\}$, $\origin{F_{i+1}}$ is a vertex-minor of $\origin{F_i}$,
	\item in $F_i$ with $i\in \{1, 2, \ldots, j\}$, 
	\begin{itemize}
	\item $B_1, B_2, \ldots, B_i$ are processed, 
	\item for $\ell\in \{2, 3, \ldots, i\}$, $\dist_{F_i}(B_\ell, F(B_\ell))\ge 1$, 
	\item if $B\in \{B_{i+1}, B_{i+2}, \ldots, B_m\}$ is a bag where $F(B)$ is processed, then 
	$\dist_{F_i}(B, F(B))\ge 3$,
	\item if $B\in \{B_{i+1}, B_{i+2}, \ldots, B_m\}$ is a bag where $F(B)$ is not processed,  
	then $\dist_{F_i}(B, F(B))\ge 4$.
      \item $\node{D_3}{R_3}\in V(T_{F_i})$ and $F_i$ is rooted at $\bag{F_i}{\node{D_3}{R_3}}$
	\end{itemize}
	\end{itemize}

	By ($\ast$), $B_1=R_3$ is processed. Thus, $F_1$ is indeed a sequence satisfying those conditions when $j=1$. 
	We claim that $j=m$. In other words, all bags in $\{\bag{D_3}{v}:v\in g(V(T'))\}$ can be sequentially processed.
	\begin{CLAIM}
	$j=m$.
	\end{CLAIM}
	\begin{clproof}
	Suppose for contradiction that $j<m$. 
	We may assume that $B_{j+1}$ is not processed in $F_j$, 
	otherwise, $F_1, F_2, \ldots, F_j, F_{j+1}$ is a longer sequence satisfying the required conditions.
	Clearly, $F(B_{j+1})$ is processed.
	The induction hypothesis for $j$ implies that
	$\dist_{F_j}(B_{j+1}, F(B_{j+1}))\ge 3$.
	Let $F(B_{j+1})=U_1-U_2- \cdots -U_y=B_{j+1}$ be the path of bags in $F_j$ from $F(B_{j+1})$ to $B_{j+1}$.
	
	We recursively apply Lemma~\ref{lem:primetwomarked1} to $U_2, U_3, \ldots, U_{y-1}$ so that the bag modified from each of $U_2, U_3, \ldots, U_{y-1}$ is a star whose center is unmarked.
	Note that when we apply Lemma~\ref{lem:primetwomarked1} to $U_2, U_3, \ldots, U_{y-1}$, 
	the decomposition tree does not change.
	
	After then, we apply Lemma~\ref{lem:primetwomarked2} to $B_{j+1}$ so that the bag modified from $B_{j+1}$ is a star whose center is unmarked.
	When we apply Lemma~\ref{lem:primetwomarked2} to $B_{j+1}$, some child bags of $B_{j+1}$ may be merged with $B_{j+1}$. 
	Thus if $U$ is a bag with $F(U)=B_{j+1}$, then the value $\dist_{F_j}(U, B_{j+1})$ may decrease by at most $1$.
	
	Let $F_{j+1}$ be the resulting decomposition. 
	We can verify that  in $F_{j+1}$, 
		\begin{itemize}
	\item $B_1, B_2, \ldots, B_j, B_{j+1}$ are processed, 
	\item for $\ell\in \{2, 3, \ldots, j+1\}$, $\dist_{F_{j+1}}(B_\ell, F(B_\ell))\ge 1$, 
	\item if $B\in \{B_{j+2}, B_{j+3}, \ldots, B_m\}$ is a bag where $F(B)$ is processed, then 
	$\dist_{F_{j+1}}(B, F(B))\ge 3$,
	\item if $B\in \{B_{j+2}, B_{j+3}, \ldots, B_m\}$ is a bag where $F(B)$ is not processed,  
	then $\dist_{F_{j+1}}(B, F(B))\ge 4$.
      \item $\node{D_3}{R_3}\in V(T_{F_i})$ and $F_i$ is rooted at $\bag{F_i}{\node{D_3}{R_3}}$
	\end{itemize}
	This contradicts the maximality of the sequence.
	We conclude that $j=m$.
	\end{clproof}
	Let $D_4:=F_m$. Note that $T_{D_4}$ is isomorphic to a subdivision of $T'$, and every bag of $D_4$ is a star whose center is unmarked.
	Therefore, $\origin{D_4}$ is isomorphic to a tree that can be obtained from a subdivision of $T'$ by adding some leaves, and in particular, 
	$\origin{D_4}$ contains an induced subgraph isomorphic to a subdivision of $T'$. Thus, $G$ contains a vertex-minor isomorphic to $T'$, and also contains a vertex-minor isomorphic to $T$, as required.
	\end{proof}

\section{Distance-hereditary vertex-minor obstructions for graphs of bounded linear rank-width}\label{sec:obstructions}

	In this section, we describe a way to generate all vertex-minor obstructions for graphs of bounded linear rank-width that are distance-hereditary graphs.  
	It generalizes the constructions developed by Jeong, Kwon, and Oum~\cite{JKO2014}. 

	For a distance-hereditary graph $G$, a connected distance-hereditary graph $G'$ is a \emph{one-vertex DH-extension} of $G$ if $G=G'- v$ 
	for some vertex $v\in V(G')$.  For convenience, if
	$G'$ is a \emph{one-vertex DH-extension} of $G$, and $D$ and $D'$ are canonical split decompositions of $G$ and $G'$ respectively, 
	then $D'$ is also called a \emph{one-vertex DH-extension} of $D$.

        Let $D_1,D_2$ and $D_3$ be three canonical split decompositions. For each $i\in \{1,2,3\}$, let $D_i'$ be a one-vertex DH extension of $D_i$ with a new unmarked
        vertex $w_i$ and such that $w_i$ is not contained in a star bag centered at $w_i$. 
        Furthermore, we choose an unmarked vertex $z_i$ linked to $w_i$.
        Let $B$ be a complete graph or a star, on three vertices
        $v_1,v_2,v_3$. For each $i\in \{1, 2,3\}$, let $D_i''$ be a split decomposition such that
        \begin{enumerate} 
        \item if $B$ is a complete graph, then $D_i'':=D_i'*w_i$,
        \item if $B$ is a star with center $v_i$, then $D_i'':=D_i'\wedge w_iz_i$,
        \item if $B$ is a star with $v_i$ a leaf, then $D_i'':=D_i'$.
        \end{enumerate} 
        We let $\mathcal{N}(D_1,D_2,D_3, K)$ be the set of all possible canonical split decompositions 
        obtained from the disjoint union of such $D_1'',D_2'', D_3''$ and a complete bag $B$ on three vertices $v_1, v_2, v_3$, by adding the marked
        edges $v_1w_1, v_2w_2,$ and $v_3w_3$.
        For $i\in \{1,2,3\}$, we let $\mathcal{N}(D_1,D_2,D_3, (S, i))$ 
        be the set of all possible canonical split decompositions obtained from the disjoint union of such $D_1'',D_2'', D_3''$ and a star bag $B$ on three vertices $v_1, v_2, v_3$ whose center is $v_i$, by adding the marked
        edges $v_1w_1, v_2w_2,$ and $v_3w_3$.  

        For a set $\mathcal{D}$ of canonical split decompositions, we let 
        \begin{align*}
          \Delta(\mathcal{D}) & := \{\mathcal{N}(D_1,D_2,D_3, K)\mid D_1,D_2,D_3\in \cD\} \\
          							&\cup \{\mathcal{N}(D_1,D_2,D_3, (S,i))\mid D_1,D_2,D_3\in \cD, i\in \{1,2,3\}\},\\
          \mathcal{D}^{+}&:=\mathcal{D}\cup \{D':D' \text{ is a one vertex  DH-extension of }D\in \mathcal{D}\}.
        \end{align*}

        For each non-negative integer $k$, we recursively construct the set $\Psi_k$ of canonical split decompositions as follows.

\begin{enumerate}
\item $\Psi_0:=\{K_2\}$ ($K_2$ is the canonical split decomposition of itself.)
\item For $k\ge 0$, let $\Psi_{k+1}:=\Delta(\Psi_k^{+})$.
\end{enumerate}

We prove the following.

\begin{theorem}\label{thm:mainobs} 
	Let $k$ be a non-negative integer. 
  Every distance-hereditary graph of linear rank-width at least $k+1$ contains a vertex-minor isomorphic to a graph whose
  canonical split decomposition is isomorphic to a decomposition in $\Psi_k$.
\end{theorem}

	We prove some intermediate lemma.
	
	\begin{figure}
\centerline{
\includegraphics[scale=0.35]{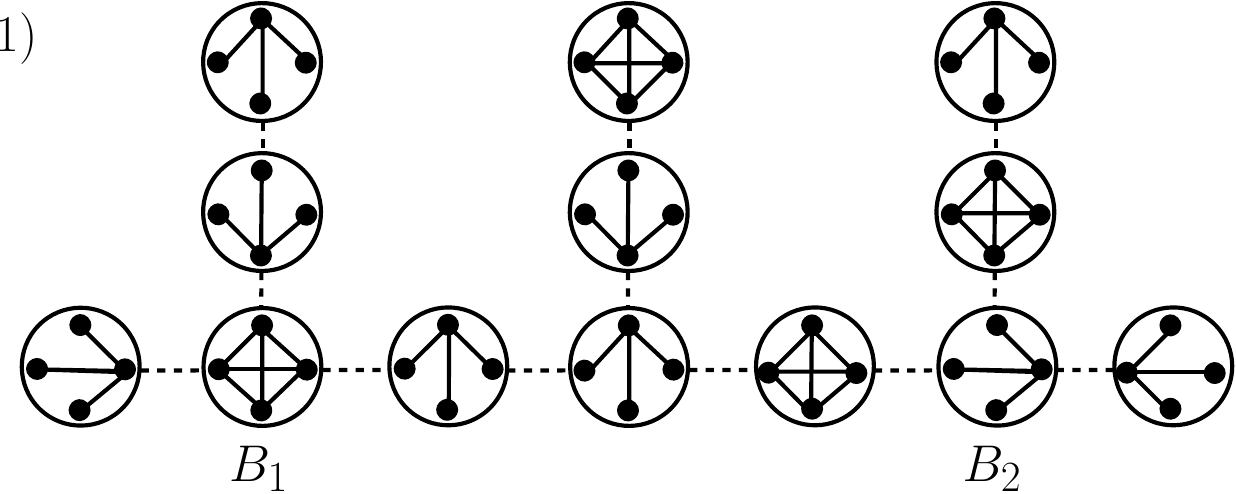} \quad\quad\quad
\includegraphics[scale=0.35]{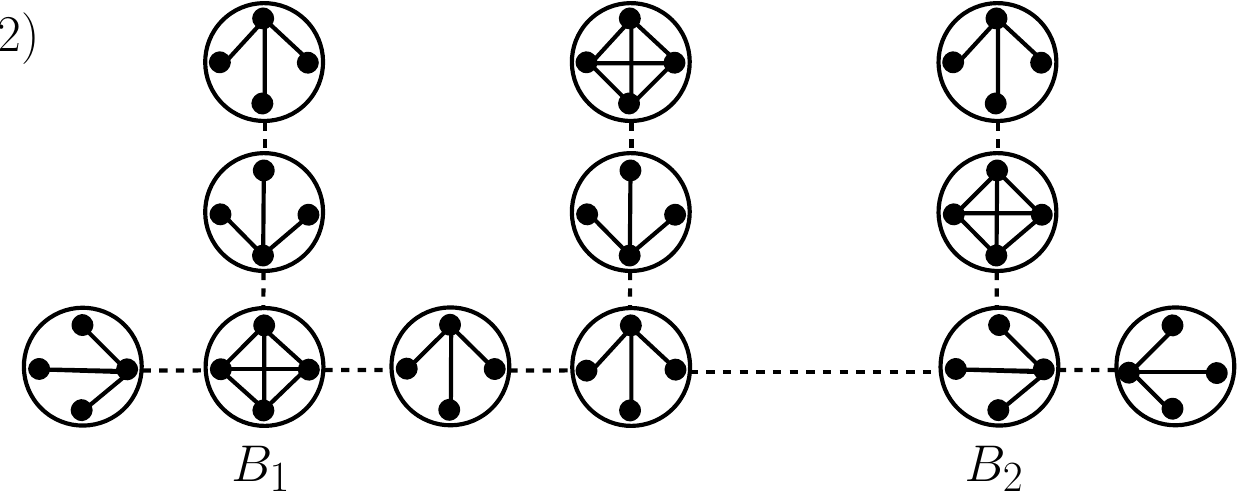} }
\vskip 0.7cm
\centerline{
\includegraphics[scale=0.35]{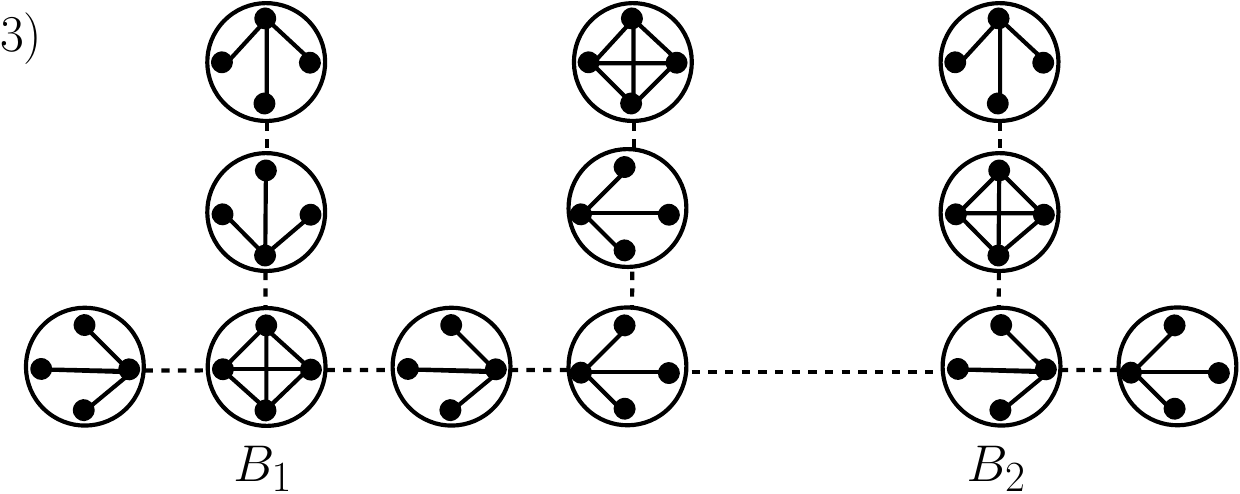} \quad\quad\quad
\includegraphics[scale=0.35]{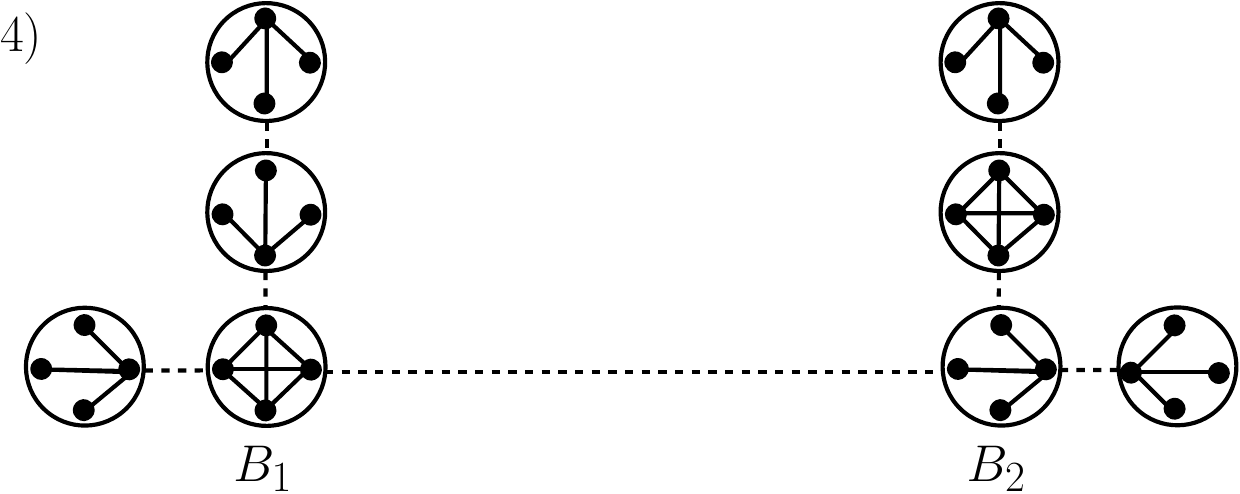} }
\caption{A shorten procedure described in Lemma~\ref{lem:reduce}. }
\label{fig:shorten}
\end{figure}

\begin{lemma}\label{lem:reduce}
  	Let $D$ be the canonical split decomposition of a connected distance-hereditary graph containing two distinct bags $B_1$ and $B_2$,  
  	and for each $i\in \{1,2\}$, let $T_i$ be the connected component of $D- V(B_i)$ such that $T_i$ contains $B_{3-i}$.
  If 
  \begin{itemize}
  \item $\zeta_b(D, B_1, T_1)$ is not the center of a star and 
  \item $B_2$ is a star bag and $\zeta_b(D, B_2, T_2)$ is a leaf of $B_2$,
  \end{itemize}
  then there exists a canonical split decomposition $D'$ such that
  \begin{enumerate}
  \item  $\origin{D}$ has $\origin{D'}$ as a vertex-minor,
  \item $D[V(T_2)\setminus V(T_1)]=D'[V(T_2)\setminus V(T_1)]$, 	
  \item $D[V(T_1)\setminus V(T_2)]=D'[V(T_1)\setminus V(T_2)]$, and
  \item either $B_1$ and $B_2$ are adjacent in $D'$, or 
    there is a path of bags $B_1-B-B_2$ in $D'$ such that $\abs{V(B)}=3$ and $B$ is a star bag whose center is unmarked.
  \end{enumerate}
\end{lemma}

\begin{proof}
  If $B_1$ and $B_2$ are adjacent bags in $D$, then we are done.  We assume that $B_1$ and $B_2$ are not adjacent.
  Let $B_1=U_1-U_2- \cdots -U_m=B_2$ be the path of bags in $D$. 
	Also, let $P=p_1p_2 \ldots p_{\ell}$ be the shortest path from $\zeta_b(D, B_1, T_1)=p_1$ to
  $\zeta_b(D, B_2, T_2)=p_{\ell}$ in $D$. Note that $\ell\ge 4$ as $m\ge 3$.
  
  Suppose there exists a bag $U_i$ containing exactly two consecutive vertices $p_j$, $p_{j+1}$ of $P$. 
  In this case, we remove $U_i$ and remove all the connected components of $D- V(U_i)$ that contain neither $B_1$ nor $B_2$, 
  and add a marked edge $p_{j-1}p_{j+2}$.
  This procedure corresponding to removing all unmarked vertices in the removed sub-decomposition.
  Since this operation does not change the parts $D[V(T_2)\setminus V(T_1)]$ and $D[V(T_1)\setminus V(T_2)]$,
  applying this operation consecutively, we may assume that 
  for $i\in \{2, 3, \ldots, m-1\}$, $U_i$ contains three consecutive vertices of $P$. 
  In other words, $U_i$ is a star whose center is adjacent to neither a vertex of $U_{i-1}$ nor to a vertex of $U_{i+1}$. 
  See 2) of Figure~\ref{fig:shorten}.

  Suppose $m\ge 4$. Note that $U_2$ contains $p_2, p_3, p_4$ and $U_3$ contains $p_5, p_6, p_7$.  Take two unmarked vertices $x_3$ and $x_6$ of $D$ that are represented
  by $p_3$ and $p_6$, respectively.  Observe that $x_3$ and $x_6$ are linked in $D$.  Let $D':=D\wedge x_3x_6$. Notice that $D'[V(U_2)]$ and $D'[V(U_3)]$ are stars whose
  centers are adjacent to each other.  Moreover, $D'[V(T_2)\setminus V(T_1)] = D[V(T_2)\setminus V(T_1)]$ and similarly,
  $D'[V(T_1)\setminus V(T_2)] = D[V(T_1)\setminus V(T_2)]$.  For each $i\in \{2,3\}$, we delete from $D'$, $U_i$ and all the connected components
  of $D' - V(U_i)$, except two connected components containing $B_1$ and $B_2$ respectively, and add the marked edge $p_1p_8$.  
    See 3) and 4) of Figure~\ref{fig:shorten}.
  By the assumption that $p_1$ is not the center of $B_1$, the
  marked edge incident with $B_1$ is of type $S_pS_p$ or $KS_p$.  Therefore, the resulting decomposition is a canonical split decomposition satisfying the conditions (1), (2), (3),
  and the number of bags containing $P$ is decreased by two. 
  
  Applying this procedure recursively, at the end, we obtain a canonical split decomposition such that
  either $B_1$ and $B_2$ are adjacent, or 
    there is a path of bags $B_1-B-B_2$ such that $B$ is a star bag whose center is adjacent to neither $B_1$ nor $B_2$.
    In the latter case, we remove all unmarked leaves of $B$, and remove all connected components of $D-V(B)$ containing neither $B_1$ nor $B_2$, 
    and replace the center of $B$ with an unmarked vertex represented by it. 
    Then we obtain the required decomposition.
\end{proof}

	The next proposition says how we can replace limbs having linear rank-width $k\ge 1$ into a canonical split decomposition in $\Psi_{k-1}^{+}$ using Lemma~\ref{lem:reduce}.
	In this proposition, we sometimes remove unmarked vertices from a given split decomposition, to take a split decomposition of the graph obtained by removing the corresponding vertices.
	We described this operation in Section~\ref{subsec:modifications}.

\begin{proposition}\label{prop:replace}   
  Let $D$ and $A$ be the canonical split decompositions of some connected distance-hereditary graphs. 
  Let $B$ be a star bag of $D$ and $v$ be a leaf of $B$, and $T$ be a connected component of $D- V(B)$ such that $\zeta_b(D, B, T)=v$, and 
  let $w$ be an unmarked
  vertex of $D$ represented by $v$.  
  If $\limbhat_D[B,w]$ has a vertex-minor that is either $\origin{A}$ or a one-vertex
  DH extension of $\origin{A}$, then there exists a canonical split decomposition $D'$, a vertex-minor of $D$, such that
  \begin{enumerate}
  \item either $D'- V(T)=D- V(T)$ or $D'- V(T)=(D- V(T))*v$, and
  \item for some unmarked vertex $w'$ of $D'$ represented by $v$, $\limbtil_{D'}[B,w']$ is either $A$ or a one-vertex DH-extension of $A$.
  \end{enumerate}
\end{proposition}

\begin{proof} 
	Suppose $\limbhat_D[B,w]$ has a vertex-minor that is either $\origin{A}$ or a one-vertex
  DH extension of $\origin{A}$. It means that
	  there exist a sequence $x_1, x_2, \ldots, x_m$ of vertices of $\limbhat_D[B,w]$ and $S\subseteq V(\limbhat_D[B,w])$ such that 
	  $	(\limbhat_D[B,w]*x_1*x_2* \ldots *x_m)- S$ is
  	either $\origin{A}$ or a one-vertex DH-extension of $\origin{A}$. 
  	So, there exists $Q \subseteq V(\mathcal{L}_D[B,w])$ such that the graph obtained from $(\limb_D[B,w]*x_1*x_2* \ldots *x_m)[Q]$ by recomposing all marked edges is 
	either $\origin{A}$ or a one-vertex 	DH-extension of $\origin{A}$.  
  	As $v$ is a leaf of $B$, $\limb_D[B,w]$ is an induced subgraph of $D$.
  	Thus, we have
  	\[(\limb_D[B,w]*x_1*x_2* \ldots *x_m)[Q]=(D*x_1*x_2* \ldots *x_m)[Q].\] 
	Let $D_1=D*x_1*x_2* \ldots *x_m$.  Note that 	$D[V(B)]=D_1[V(B)]$ as $v$ is a leaf of $B$, and $\{x_1, x_2, \ldots, x_m\}\subseteq V(T)$.

  We choose a bag $B'$ in $D_1$ such that
  \begin{enumerate}
  \item $B'$ has a vertex of $Q$, and 
  \item $\dist_{D_1}(B, B')$ is minimum.
  \end{enumerate}
  Let us check that all the hypothesis of Lemma \ref{lem:reduce} with $(B_1, B_2)=(B', B)$ are satisfied.  Let $T_1$ be the connected component of $D_1- V(B')$ containing $B$ and let $T_2$ be the
  connected component of $D_1- V(B)$ containing $B'$.  Let $y:=\zeta_b(D_1, B', T_1)$.  From the choice of $B'$, we have $y\notin Q$; otherwise, there exists an unmarked
  vertex represented by $y$, and all vertices on the path from $y$ to it should be contained in $Q$, as $Q$ induces a connected graph.  In particular, the bag in $T_1$
  containing a vertex adjacent to a marked vertex in $B'$ should contain a vertex of $Q$, and this contradicts to the minimality of the distance between $B$ and $B'$.
  In addition, $y$ is not the center of a star bag because $D_1[Q]$ is connected and $B'$ has at least two vertices of $Q$. Therefore, the bags $B$ and $B'$ satisfy the
  hypothesis of Lemma \ref{lem:reduce} with $(B_1, B_2)=(B', B)$. 
  
  By applying Lemma~\ref{lem:reduce} on $B$ and $B'$, there exists a canonical split decomposition $D_2$ such that
  \begin{enumerate}
  \item $\origin{D_1}$ has $\origin{D_2}$ as a vertex-minor,
  \item $D_1[V(T_2)\setminus V(T_1)]=D_2[V(T_2)\setminus V(T_1)]$,
  \item $D_1[V(T_1)\setminus V(T_2)]=D_2[V(T_1)\setminus V(T_2)]$,
  \item either $B$ and $B'$ are adjacent in $D_2$, or there exists a path of bags $B-B_s-B'$ in $D_2$ such that $\abs{V(B_s)}=3$ and $B_s$ is
    a star bag whose center is unmarked.
  \end{enumerate}

  We obtain $D_3$ from $D_2$ by removing the vertices of $V(T_2)\setminus V(T_1)$ that are not contained in $Q\cup \{y\}$, 
  and then recomposing all new recomposable marked edges.  
  Since recomposable marked edges only appeared in the part $V(T_2)\setminus V(T_1)$, 
  we have $D_3[V(T_1)\setminus V(T_2)]=D_2[V(T_1)\setminus V(T_2)]$. 
  Furthermore, the bag $B_s$ still exists in $D_3$ if it exists in $D_2$. This is because 
  \begin{itemize}
  \item the bag $B'$ contains at least two vertices of $Q$ in $D_2$, and thus $B'$ remains as a bag of same type in $D_3$, and
  \item the type of the marked edge connecting $B'$ and $B_s$ does not change when recompositions are applied.
  \end{itemize}
   Let $B_2$ be the bag of $D_3$ containing $y$. 
  We divide into cases depending on whether $B$ and $B_2$ are adjacent or not.

  \vskip 0.2cm
  \noindent\emph{\textbf{Case 1.} $B$ and $B_2$ are adjacent in $D_3$.}
  
  In this case, $D_3$ itself is a required decomposition. 
  Choose an unmarked vertex $z$ in $D_3$ represented by $v$. 
  Then $\limbtil_{D_3}[B, z]$ is the same as the split decomposition obtained from $(\limb_D[B,w]*x_1*x_2* \ldots *x_m)[Q]$ by recomposing all recomposable marked edges,
 which is either $\origin{A}$ or a one-vertex 	DH-extension of $\origin{A}$. 
   
  \vskip 0.2cm
  \noindent\emph{\textbf{Case 2.} There exists a path of bags $B-B_s-B_2$ such that $\abs{V(B_s)}=3$ and $B_s$ is
    a star bag whose center is unmarked.}
  
  Let $c$ be the center of $B_s$, and let $c_1$ and $c_2$ be two leaves of $B_s$ that are adjacent to $y$ and $v$, respectively.
  Choose an unmarked vertex $z$ of $D_3$ represented by $c_1$, and let $H:=\limbtil_{D_3}[B_s, z]$.
  By construction, $H$ is either $A$ or a one-vertex 	DH-extension of $A$. 

  If $H=A$, then we can regard $\limbtil_{D_3}[B, c]$ as a one-vertex DH-extension of $A$ with the new vertex $c$.  Therefore, we may assume that $H$ is a one-vertex DH-extension of $A$.
  Let $a$ be the newly added vertex $a$ in $H$.  

	We would like to remove the extended vertex $a$ from $H$, and then add $c$ to $H$ so that we obtain a new one-vertex extension of $A$ which contains $c$. 
	But this is not always possible because the operation of removing $a$ may disconnect the remaining part of $H$ from  $c$. We first deal with this special case.
  
  Assume $B_2$ is a star whose center is an unmarked vertex in $D_3$. In this case this center should be $z$.
  We obtain a new decomposition $D_4$ by applying a local complementation at $c$, removing $c$ and
  recomposing a marked edge incident with $B_s$.  Note that $D_4$ is exactly the decomposition obtained from the disjoint union of the two connected components of $D_3- V(B_s)$ by
  adding a marked edge $yv$, and thus it is canonical.  Also, $z$ is represented by $v$ in $D_4$, and we have $\limbtil_{D_4}[B, z]=H$.  Thus, $D_4$ is a required
  decomposition.

  Now we assume that $c$ is linked to at least two vertices of $H$ in $D_3$.  
  Since
  $H$ is a one vertex DH-extension of $A$ and $A$ was chosen as a canonical split decomposition of a connected graph, $\origin{H}-a$ is connected.  So, if we define $D_4$ as the canonical split decomposition obtained from $D_3- a$, then $D_4$ is connected and
  $\limbtil_{D_4}[B, c]$ can be regarded as a one vertex DH-extension of $A$. Therefore, $D_4$ is a required decomposition.
\end{proof}

\begin{proof}[Proof of Theorem \ref{thm:mainobs}] 
  We prove it by induction on $k$. 
  If $k=0$, then $\lrw(G)\ge 1$ and $G$ has an edge. 
  We may assume $k\ge 1$.
  
  Let $D$ be the canonical split decomposition of $G$.  Since $G$ has linear rank-width at least $k+1$, by Theorem~\ref{thm:mainchap2}, there exists a bag $B$ in $D$ with three connected components $T_1, T_2, T_3$ of $D- V(B)$
  such that $f_D(B, T_i)\ge k$ for each $i\in \{1,2,3\}$.

  We remove all connected components of $D-V(B)$ other than $T_1, T_2, T_3$, 
  and for each marked vertex $w$ in $B$ that was adjacent to some removed component, 
  we choose a vertex $w'$ in $D$ represented by $B$ and replace $w$ with $w'$.
  Note that the resulting decomposition is a canonical split decomposition of an induced subgraph of $G$.

  Now, if $B$ is a star whose center is unmarked, then we apply a local complementation at this vertex, and 
  otherwise, we change nothing.
  Then we obtain a new decomposition by removing all unmarked vertices in $B$.
 	Let us denote by $D'$ this canonical split decomposition and
	denote by $B'$ the bag modified from $B$. 
	
	For each $i\in \{1,2,3\}$, 
  let $v_i:=\zeta_b(D', B', T_i)$ and $w_i:=\zeta_c(D', B', T_i)$, and
  $z_i$ be an unmarked vertex of $D'$ represented by $v_i$ in $D'$.
  
  We define a new decomposition $D_1$ as follows.
	If $B'$ is a star bag centered at $v_3$, then let $D_1:=D'$.
	If $B$ is a complete bag, then let $D_1:=D'*z_3$.
	If $B$ is a star bag centered at $v_i\in \{v_1,v_2\}$, then let $D_1:=D*z_i*z_3$.  
		One easily checks that $D_1[\{v_1,v_2,v_3\}]$ is a star centered at $v_3$. Let $B^1:=D_1[\{v_1,v_2,v_3\}]$ and, for
  $j\in \{1,2,3\}$, let $T^1_j:=D_1[V(T_j)]$.  Note that $z_i$ is still represented by $v_i$.
  
  Since $v_1$ and $v_2$ are leaves of $B^1$, 
  for each $i\in \{1,2\}$, 
  $\limb_{D_1}[B^1, z_i]=T^1_i- w_i$ and 
  by the induction hypothesis,
  there exists a canonical split decomposition $F_i$ in $\Psi_{k-1}$ such that
  $\limbhat_{D_1}[B^1, z_i]$ has a vertex-minor isomorphic to $\origin{F_i}$.
  By applying Proposition~\ref{prop:replace} to $T^1_1$ and $T^1_2$, 
  we can obtain a canonical split decomposition $D_2$ satisfying that
  \begin{enumerate}
  \item $D_2[V(B^1)]=D_1[V(B^1)]$,
  \item $D_2[V(T_3)]$ is either $T_3$ or $T_3*w_3$, and
  \item for each $i\in \{1,2\}$, 
    $\limbtil_{D_2}[D_2[V(B^1)], z^2_i]$ is isomorphic to a canonical split decomposition in $\Psi_{k-1}^+$ for some unmarked vertex $z^2_i$ of $D_2$ represented by $v_i$.
  \end{enumerate}

  Let $B^2:=D_2[V(B^1)]$.  For each $i\in \{1,2\}$, let $T^2_i$ be the connected component of $D_2- V(B^2)$ containing $z^2_i$, and $w^2_i:=\zeta_c(D_2, B^2, T^2_i)$.
  Let $w^2_3:=w_3$, $z^2_3:=z_3$, and $T^2_3:=D_2[V(T_3)]$.

 Now, we want to transform $B^2$ into a star whose center is $v_2$ by applying local complementations at $z^2_3$ and $z^2_2$.
 We can verify that
  \begin{enumerate}
  \item $(D_2*z^2_3*z^2_2)[V(B^2)]$ is a star whose center is $v_2$,
  \item $(D_2*z^2_3*z^2_2)[V(T_1^2)]=T_1^2$, 
  \item $(D_2*z^2_3*z^2_2)[V(T_2^2)]=T_2^2*w_2^2*z_2^2$,
  \item $(D_2*z^2_3*z^2_2)[V(T_3^2)]=T_3^2*z_3^2*w_3^2$.
  \end{enumerate}
  We apply Proposition~\ref{prop:replace} to $D_2*z^2_3*z^2_2$ and obtain a canonical split decomposition $D_3$ so that
  \begin{enumerate}
  \item $D_3[V(B^2)]=(D_2*z_3^2*z_2^2)[V(B^2)]$ and $D_3[V(T_1^2)]=(D_2*z_3^2*z_2^2)[V(T_1^2)]$,
  \item $D_3[V(T_2^2)]$ is either $(D_2*z_3^2*z_2^2)[V(T_2^2)]$ or $(D_2*z_3^2*z_2^2)[V(T_2^2)]*w_2^2$, and
  \item 
    $\limbtil_{D_3}[D_3[V(B^2)], z_3^3]$ is isomorphic to a canonical split decomposition in $\Psi_{k-1}^+$ for some unmarked vertex $z_3^3$ of $D_3$ represented by $v_3$.
  \end{enumerate}
    Let $B^3:=D_3[V(B^2)]$.
 Let $T_3^3$ be the connected component of $D_3- V(B^3)$ containing $z_3^3$,
  and $w_3^3:=\zeta_c(D_3, B^3, T_3^3)$. 
  Note that $T_3^3- w_3^3\in \Psi_{k-1}^+$ and for $i\in \{1,2\}$, $z_i^2$ is still represented by $v_i$ in $D_3$.
  We define $T^3_1:=D_3[V(T^2_1)]$, $T^3_2:=D_3[V(T^2_2)]$ and 
  define $w^3_1, w^3_2, z^3_1, z^3_2$ as the same as 
  $w^2_1, w^2_2, z^2_1, z^2_2$, respectively.
  
  Now we claim that $D_3\in \Psi_k$ or $D_3*z_2^3\in \Psi_k$. 
  We observe two cases depending on whether $T^3_2$ is equal to $(D_2*z_3^2*z_2^2)[V(T_2^2)]$ or to $(D_2*z_3^2*z_2^2)[V(T_2^2)]*w_2^2$.  
  
  \vskip 0.2cm
  \noindent\emph{\textbf{Case 1.} 
    $T^3_2=(D_2*z_3^2*z_2^2)[V(T_2^2)]$.}
  
  We observe that
  $B^3$ is a star whose center is $v_2$, and
  the three connected components of $D_3- V(B^3)$ are
  $T_1^2$, $T_2^2*w_2^2*z_2^2$, and $T_3^3$.
  In this case, $D_3*z_2^2\in \Psi_k$ because 
  \begin{enumerate}
  \item $(D_3*z_2^2)[V(B^3)]$ is a complete bag, and
  \item the three components of $D_3- V(B^3)$ are
    $T_1^2*w_1^2$, $T_2^2*w_2^2$, and $T_3^3*w_3^3$,
  \end{enumerate}
  and the limbs of $D_3*z_2^2$ with respect to $B^3$ are 
  $T_1^2- w_1^2$, $T_2^2- w_2^2$, and $T_3^3- w_3^3$, which are contained in $\Psi_{k-1}^+$.

  \vskip 0.2cm
  \noindent\emph{\textbf{Case 2.} 
    $T^3_2=(D_2*z_3^2*z_2^2)[V(T_2^2)]*w_2^2$. }

  We observe that $B^3$ is a star centered at $v_2$, and the three components of $D_3- V(B^3)$ are $T_1^2$, $T_2^2*w_2^2*z_2^2*w_2^2=T_2^2\wedge w_2^2z_2^2$, and $T_3^3$.  We can see that
  $D_3\in \Psi_k$ because the limbs with respect to $B^3$ are $T_1^2- w_1^2$, $T_2^2- w_2^2$, and $T_3^3- w_3^3$, which are contained in $\Psi_{k-1}^+$.
  
  \vskip 0.2cm
  We conclude that $G$ has a vertex-minor isomorphic to $\origin{D_3}$ where $D_3\in \Psi_{k}$, as required.
\end{proof}

In order to prove that $\Psi_k$ is a minimal set of canonical split decompositions of distance-hereditary vertex-minor obstructions for linear rank-width at most $k$, we need to prove that for every
$D\in \Psi_k$, $\origin{D}$ has linear rank-width $k+1$ and all its proper vertex-minors have linear rank-width at most $k$. 
However, while $\lrw(\origin{D}) = k+1$ for all $D\in \Psi_k$, they are not minimal with respect to having linear rank-width $k+1$. For instance for many canonical split
decompositions $D$ in $\Psi_1$, $\origin{D}$ is not a vertex-minor obstruction for linear rank-width $1$ as it contains either $\alpha_1$ or $\gamma_1$ as a vertex-minor
(see Section \ref{sec:charlrw1}). %
We guess that the following set $\Phi_k$ would form a minimal set of distance-hereditary vertex-minor obstructions, but we leave it as an open problem.

\begin{enumerate}
\item $\Phi_0:=\{K_2\}$.
\item For $k\ge 0$, let $\Phi_{k+1}:=\Delta(\Phi_k)$.
\end{enumerate}

Our intuition is supported by the following.

\begin{proposition}\label{prop:phik}  
Let $k$ be a non-negative integer and let $D\in \Phi_k$. Then $\lrw(\origin{D}) = k+1$ and every proper vertex-minor of $\origin{D}$ has linear rank-width at most $k$. 
\end{proposition}

We need the following two lemmas.

\begin{lemma}\label{lem:locphi} Let $D\in \Phi_k$ and $v$ be an unmarked vertex in $D$.  Then $D*v\in \Phi_k$.
\end{lemma}
	
\begin{proof} 
  We proceed by induction on $k$.  We may assume that $k\ge 1$.  By the construction, there exists a bag $B$ of $D$ such that the three limbs $D_1$, $D_2$, $D_3$ in $D$ corresponding to the bag $B$
  are contained in $\Phi_{k-1}$. 
	
  Let $B':=B$ or $B':=B*v'$ be a bag of $D*v$ depending on whether $v$ has a representative $v'$ in $B$. Let $D_1'$, $D_2'$ and $D_3'$ be the three limbs of $D*v$
  corresponding to the bag $B'$ such that $D_i'$ and $D_i$ came from the same component of $D- V(B)$.  One checks by Proposition~\ref{prop:preservelrw} that $D_i'$ is locally equivalent to
  $D_i$.  So by the induction hypothesis, $D_i'\in \Phi_{k-1}$.  And $D*v$ is the canonical split decomposition obtained from $D_i'$ following the construction of
  $\Phi_k$.  Therefore, $D*v\in \Phi_k$.
\end{proof}

\begin{lemma}[Bouchet~\cite{Bouchet1988}]\label{lem:bouchet} Let $G$ be a graph, $v$ be a vertex of $G$ and $w$ be an arbitrary neighbor of $v$.  Then every elementary vertex-minor obtained from $G$ by
  deleting $v$ is locally equivalent to either $G- v$, $G* v- v$, or $G\wedge vw- v$.
\end{lemma}

\begin{proof}[Proof of Proposition \ref{prop:phik}] By construction, it is not hard to prove by induction with the help of Theorem \ref{thm:mainchap2} that $\lrw(\origin{D})=k+1$ for every
  split decomposition $D\in \Phi_k$. For the second statement, by Lemmas~\ref{lem:locphi} and \ref{lem:bouchet}, it is sufficient to show that if $D\in \Phi_k$ and $v$ is an unmarked vertex of $D$,
  then $\origin{D}- v$ has linear rank-width at most $k$.  We use induction on $k$ to prove it. We may assume that $k\ge 1$.  Let $B$ be the bag of $D$ such that $D- V(B)$ has exactly
  three limbs that are contained in $\Phi_{k-1}$.  Clearly there is no other bag having the same property.  Since $B$ has no unmarked vertices, $v$ is contained in one of the limbs
  $D'$, and by induction hypothesis, $\origin{D'}- v$ has linear rank-width at most $k-1$.  Therefore, by Theorem~\ref{thm:mainchap2}, $\origin{D}- v$ has linear rank-width at most
  $k$.
\end{proof}

We finish by pointing out that it is proved in \cite{JKO2014} that the number of distance-hereditary vertex-minor obstructions for linear rank-width $k$ is at least $2^{\Omega(3^k)}$. One can easily
check by induction that the number of graphs in $\Phi_k$ is bounded by $2^{O(3^k)}$. Therefore, we can conclude that the number of distance-hereditary vertex-minor obstructions for linear rank-width
$k$ is equal to $2^{\theta(3^k)}$. 

\section{Simpler proofs for the characterizations of graphs of linear rank-width at most $1$}\label{sec:charlrw1}

In this section, we obtain simpler proofs for known characterizations of the graphs of linear rank-width at most $1$ using Theorem~\ref{thm:mainchap2}.  Theorem~\ref{thm:charlrw1} was originally
proved by Bui-Xuan, Kant\'{e}, and Limouzy~\cite{Bui-XuanKL13}.

\begin{theorem}[Bui-Xuan, Kant\'{e}, and Limouzy~\cite{Bui-XuanKL13}]\label{thm:charlrw1}
  Let $G$ be a connected graph and let $D$ be the canonical split decomposition of $G$.
  The following two are equivalent.
  \begin{enumerate}[(1)]
  \item $G$ has linear rank-width at most $1$.
  \item $G$ is distance-hereditary and $T_D$ is a path.
  \end{enumerate}
\end{theorem}
\begin{proof}
  We first prove that (2) implies (1). Let $T_D:=u_1u_2 \cdots u_m$.  For each $1\le i\le m$, we take any ordering $L_i$ of unmarked vertices in $\bag{D}{u_i}$.  
  Since $G$ is distance-hereditary, by Theorem~\ref{thm:Bouchet88}, each bag of $D$ is a complete graph or a star.
  Thus, we can easily check that
  $L_1\oplus L_2\oplus \ldots \oplus L_m$ is a linear layout of $G$ having width at most $1$.

  We prove that (1) implies (2). Suppose $G$ has linear rank-width at most $1$. 
  From the known fact that a connected graph has rank-width at most $1$ if and only if it is distance-hereditary~\cite{Oum05}, 
  $G$ is distance-hereditary. 
  Suppose $T_D$ is not a path. Then there exists a bag $B$ of $D$ such that 
  $B$ has at least three neighbor bags in $D$.
  Thus, 
  $D- V(B)$ has at least three components $T$ where $f_D(B, T)\ge 1$.
  By Theorem~\ref{thm:mainchap2}, $G$ has linear rank-width at least $2$, which is a contradiction.
\end{proof}

From Theorem~\ref{thm:charlrw1},
we have a linear-time algorithm to recognize the graphs of linear rank-width at most $1$.

\begin{theorem}\label{thm:recoglrw1}
  For a given graph $G$, 
  we can test whether $G$ has linear rank-width at most $1$ or not in time $\mathcal{O}(\abs{V(G)}+\abs{E(G)})$.
\end{theorem}
\begin{proof}
  We first compute the canonical split decomposition $D$ of each connected component of $G$ using the algorithm from Theorem~\ref{thm:CED}. It takes $\mathcal{O}(\abs{V(G)}+\abs{E(G)})$ time.
  Furthermore, this algorithm outputs the type of each bag together.
   Note that each bag of a canonical split decomposition of a connected distance-hereditary graph is either a complete graph or a star by Theorem~\ref{thm:Bouchet88}.
	Thus, if there is a prime bag, then we answer that $G$ has linear rank-width more than $1$.
 
  Additionally, we check whether $T_D$ is a path or not.
  By Theorem~\ref{thm:charlrw1}, if $T_D$ is a path and each bag is not prime, then we conclude that $G$ has linear rank-width at most $1$, and otherwise, $G$ has linear rank-width at least $2$.
\end{proof}

The list of induced subgraph obstructions for graphs of linear rank-width at most $1$ was characterized by Adler, Farley, and Proskurowski~\cite{AdlerFP11}.
The obstructions consist of the known obstructions for distance-hereditary graphs~\cite{BandeltM86}, and the set $\obt$ of the induced subgraph obstructions for graphs of linear rank-width at most $1$ that are distance-hereditary. See Figure~\ref{fig:obslrw1} for the list of obstructions $\alpha_i, \beta_j, \gamma_k$ in $\obt$ where $1\le i\le 4$, $1\le j\le 6$, $1\le k\le 4$.
This set  $\obt$ can be obtained from Theorem~\ref{thm:charlrw1} in a much easier way than the previous result.

	A graph $H$ is called a
	\emph{pivot-minor} of a graph $G$ if $H$ can be obtained from $G$ by applying a sequence of pivoting on  edges and deletions of vertices.

\begin{figure}
  \tikzstyle{v}=[circle, draw, solid, fill=black, inner sep=0pt, minimum width=3pt]
  \centering
  \begin{tikzpicture}[scale=0.5]
    \node[v](v1) at (0,.8){};
    \node[v](v2) at (0,1.6){};
    \node[v](v3) at (-.8,-.5){};
    \node[v](v4) at (.8,-.5){};
    \node[v](v5) at (-1.6,-1){};
    \node[v](v6) at (1.6,-1){};
    \draw (v1)--(v3)--(v4)--(v1);
    \draw (v1)--(v2);
    \draw (v3)--(v5);
    \draw (v4)--(v6);
    \node [label=$\alpha_1$] at (0,-2.7) {};
  \end{tikzpicture}\quad
  \begin{tikzpicture}[scale=0.5]
    \node[v](v1) at (0,2){};
    \node[v](v2) at (1-.2,2){};
    \node[v](v3) at (1.5,1+.2){};
    \node[v](v4) at (1.5,3-.2){};
    \node[v](v5) at (2+.2,2){};
    \node[v](v6) at (3,2){};
    \draw (v1)--(v2)--(v3)--(v5)--(v6);
    \draw (v2)--(v4)--(v5);
    \draw (v2)--(v5);
    \node [label=$\alpha_2$] at (1.5,0-.7) {};
  \end{tikzpicture}\quad
  \begin{tikzpicture}[scale=0.5]
    \node[v](v2) at (1-.2-.6,2){};
    \node[v](v3) at (1.5,1+.2-.6){};
    \node[v](v4) at (1.5,3-.2+.6){};
    \node[v](v5) at (2+.2+.6,2){};
    \node[v](v6) at (1.5,2){};
    \node[v](v7) at (1,2.5){};
    \draw (v2)--(v3)--(v4)--(v5)--(v2);
    \draw (v2)--(v4);
    \draw (v3)--(v5);
    \draw (v6)--(v7);
    \node [label=$\alpha_3$] at (1.5,0-.7) {};
  \end{tikzpicture}\quad
  \begin{tikzpicture}[scale=0.5]
    \node[v](v2) at (0.2,2){};
    \node[v](v3) at (0.7-.3,1.5){};
    \node[v](v4) at (2.3+.3,2.5){};
    \node[v](v5) at (2.8,2){};
    \node[v](v6) at (1.5,3){};
    \node[v](v7) at (1.5,1){};
    \draw (v2)--(v4)--(v5)--(v3)--(v2);
    \draw (v6)--(v2);\draw (v6)--(v3);\draw (v6)--(v4);\draw (v6)--(v5);
    \draw (v7)--(v2);\draw (v7)--(v3);\draw (v7)--(v4);\draw (v7)--(v5);
    \node [label=$\alpha_4$] at (1.5,0-.7) {};
  \end{tikzpicture}
  
  \begin{tikzpicture}[scale=0.5]
    \node[v](v1) at (0,2){};
    \node[v](v2) at (1-.2,2){};
    \node[v](v3) at (1.5,1+.2){};
    \node[v](v4) at (1.5,3-.2){};
    \node[v](v5) at (2+.2,2){};
    \node[v](v6) at (3,2){};
    \draw (v1)--(v2)--(v3)--(v5)--(v6);
    \draw (v2)--(v4)--(v5);
    \node [label=$\beta_1$] at (1.5,0-.5) {};
  \end{tikzpicture}\quad
  \begin{tikzpicture}[scale=0.5]
    \node[v](v1) at (0,2){};
    \node[v](v2) at (1-.2,2){};
    \node[v](v3) at (1.5,1+.2){};
    \node[v](v4) at (1.5,3-.2){};
    \node[v](v5) at (2+.2,2){};
    \node[v](v6) at (3,2){};
    \draw (v1)--(v2)--(v3)--(v5)--(v6);
    \draw (v2)--(v4)--(v5);
    \draw (v3)--(v1)--(v4);
    \node [label=$\beta_2$] at (1.5,0-.5) {};
  \end{tikzpicture}\quad
  \begin{tikzpicture}[scale=0.5]
    \node[v](v1) at (0,2){};
    \node[v](v2) at (1-.2,2){};
    \node[v](v3) at (1.5,1+.2){};
    \node[v](v4) at (1.5,3-.2){};
    \node[v](v5) at (2+.2,2){};
    \node[v](v6) at (3,2){};
    \draw (v1)--(v2)--(v3)--(v5)--(v6);
    \draw (v2)--(v4)--(v5);
    \draw (v3)--(v1)--(v4);
    \draw (v3)--(v6)--(v4);
    \node [label=$\beta_3$] at (1.5,0-.5) {};
  \end{tikzpicture}
  
  \begin{tikzpicture}[scale=0.5]
    \node[v](v1) at (0,2){};
    \node[v](v2) at (1-.2,2){};
    \node[v](v3) at (1.5,1+.2){};
    \node[v](v4) at (1.5,3-.2){};
    \node[v](v5) at (2+.2,2){};
    \node[v](v6) at (3,2){};
    \draw (v1)--(v2)--(v3)--(v5)--(v6);
    \draw (v2)--(v4)--(v5);
    \draw(v3)--(v4);
    \node [label=$\beta_4$] at (1.5,0-.5) {};
  \end{tikzpicture}\quad
  \begin{tikzpicture}[scale=0.5]
    \node[v](v1) at (0,2){};
    \node[v](v2) at (1-.2,2){};
    \node[v](v3) at (1.5,1+.2){};
    \node[v](v4) at (1.5,3-.2){};
    \node[v](v5) at (2+.2,2){};
    \node[v](v6) at (3,2){};
    \draw (v1)--(v2)--(v3)--(v5)--(v6);
    \draw (v2)--(v4)--(v5);
    \draw (v3)--(v1)--(v4);
    \draw(v3)--(v4);
    \node [label=$\beta_5$] at (1.5,0-.5) {};
  \end{tikzpicture}\quad
  \begin{tikzpicture}[scale=0.5]
    \node[v](v1) at (0,2){};
    \node[v](v2) at (1-.2,2){};
    \node[v](v3) at (1.5,1+.2){};
    \node[v](v4) at (1.5,3-.2){};
    \node[v](v5) at (2+.2,2){};
    \node[v](v6) at (3,2){};
    \draw (v1)--(v2)--(v3)--(v5)--(v6);
    \draw (v2)--(v4)--(v5);
    \draw (v3)--(v1)--(v4);
    \draw (v3)--(v6)--(v4);
    \draw(v3)--(v4);
    \node [label=$\beta_6$] at (1.5,0-.5) {};
  \end{tikzpicture}
  
  \begin{tikzpicture}[scale=0.5]
    \node[v](v1) at (0,1){};
    \node[v](v2) at (-.8,1.8){};
    \node[v](v3) at (-1.6,.8){};
    \node[v](v4) at (.8,1.8){};
    \node[v](v5) at (1.6,.8){};
    \node[v](v6) at (.6,.2){};
    \node[v](v7) at (-.2,-.6){};
    \draw (v1)--(v2)--(v3);
    \draw (v1)--(v4)--(v5);
    \draw (v1)--(v6)--(v7);
    \node [label=$\gamma_1$] at (0,-2.5) {};
  \end{tikzpicture}\quad
  \begin{tikzpicture}[scale=0.5]
    \node[v](v1) at (0,1){};
    \node[v](v2) at (-.8,1.8){};
    \node[v](v3) at (-1.6,.8){};
    \node[v](v4) at (.8,1.8){};
    \node[v](v5) at (1.6,.8){};
    \node[v](v6) at (.6,.2){};
    \node[v](v7) at (-.2,-.6){};
    \draw (v1)--(v2)--(v3);
    \draw (v1)--(v4)--(v5);
    \draw (v1)--(v6)--(v7);
    \draw (v1)--(v3);
    \node [label=$\gamma_2$] at (0,-2.5) {};
  \end{tikzpicture}\quad
  \begin{tikzpicture}[scale=0.5]
    \node[v](v1) at (0,1){};
    \node[v](v2) at (-.8,1.8){};
    \node[v](v3) at (-1.6,.8){};
    \node[v](v4) at (.8,1.8){};
    \node[v](v5) at (1.6,.8){};
    \node[v](v6) at (.6,.2){};
    \node[v](v7) at (-.2,-.6){};
    \draw (v1)--(v2)--(v3);
    \draw (v1)--(v4)--(v5);
    \draw (v1)--(v6)--(v7);
    \draw (v1)--(v3);
    \draw (v1)--(v5);
    \node [label=$\gamma_3$] at (0,-2.5) {};
  \end{tikzpicture}\quad
  \begin{tikzpicture}[scale=0.5]
    \node[v](v1) at (0,1){};
    \node[v](v2) at (-.8,1.8){};
    \node[v](v3) at (-1.6,.8){};
    \node[v](v4) at (.8,1.8){};
    \node[v](v5) at (1.6,.8){};
    \node[v](v6) at (.6,.2){};
    \node[v](v7) at (-.2,-.6){};
    \draw (v1)--(v2)--(v3);
    \draw (v1)--(v4)--(v5);
    \draw (v1)--(v6)--(v7);
    \draw (v1)--(v3);
    \draw (v1)--(v5);
    \draw (v1)--(v7);
    \node [label=$\gamma_4$] at (0,-2.5) {};
  \end{tikzpicture}  

  \caption{The induced subgraph obstructions for graphs of linear rank-width at most $1$ that are distance-hereditary.}
  \label{fig:obslrw1}
\end{figure}
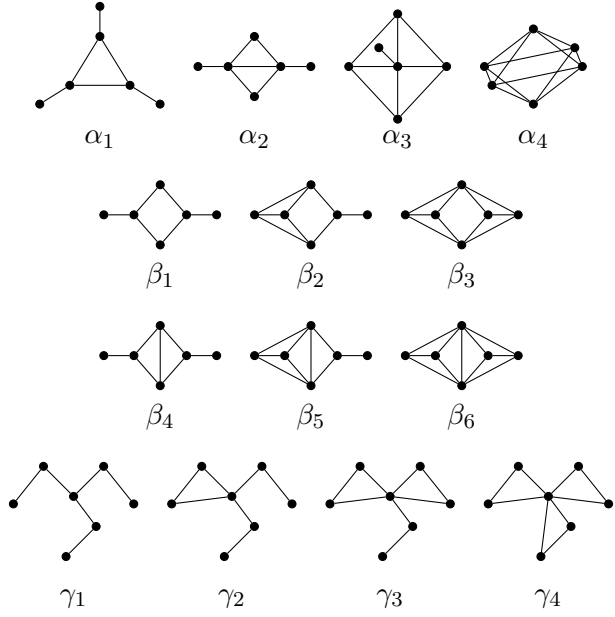

\begin{table} 
  \begin{center}
    \begin{tabular}[t]{|c| c| c| c| c|}
      \hline
      type of $B$ & type of $v_1w_1$ & type of $v_2w_2$ & type of $v_3w_3$ & induced subgraph \\ 
      \hline
      A complete bag& $KS_p$ & $KS_p$ & $KS_p$ & $\alpha_1$ \\
                  & $KS_c$ & $KS_p$ & $KS_p$ & $\alpha_2$ \\
                  & $KS_c$ & $KS_c$ & $KS_p$ & $\alpha_3$ \\
                  & $KS_c$ & $KS_c$ & $KS_c$ & $\alpha_4$ \\
      \hline
      A star bag & $S_cS_c$ & $S_pS_p$ & $S_pS_p$ & $\beta_1$ \\
      with center at $v_1$ & $S_cS_c$ & $S_pS_p$ & $S_pK$ & $\beta_2$ \\
                  & $S_cS_c$ & $S_pK$ & $S_pK$ & $\beta_3$ \\
                  & $S_cK$ & $S_pS_p$ & $S_pS_p$ & $\beta_4$ \\
                  & $S_cK$ & $S_pS_p$ & $S_pK$ & $\beta_5$\\
                  & $S_cK$ & $S_pK$ & $S_pK$ & $\beta_6$\\
      \hline
      A star bag & $S_pS_p$ & $S_pS_p$ & $S_pS_p$ & $\gamma_1$ \\
      with center at  & $S_pK$ & $S_pS_p$ & $S_pS_p$ & $\gamma_2$\\
      a vertex & $S_pK$ & $S_pK$ & $S_pS_p$ & $\gamma_3$\\
      other than $v_i$ & $S_pK$ & $S_pK$ & $S_pK$ & $\gamma_4$ \\
      \hline
    \end{tabular}
  \end{center}
  \caption{Summary of all cases in Theorem~\ref{thm:charlrw2}}\label{table2}
\end{table}

\begin{theorem}[Adler, Farley, and Proskurowski~\cite{AdlerFP11}]\label{thm:charlrw2}
  Let $G$ be a connected graph.
  The following are equivalent.
  \begin{enumerate}
  \item $G$ has linear rank-width at most $1$.
  \item $G$ is distance-hereditary and $G$ has no induced subgraph isomorphic to a graph in 
  \[\{\alpha_1, \alpha_2, \alpha_3, \alpha_4, \beta_1, \beta_2, \beta_3, \beta_4, \beta_5, \beta_6, \gamma_1, \gamma_2, \gamma_3, \gamma_4\}.\]
  \item $G$ has no pivot-minor isomorphic to a graph in $\{C_5, C_6, \alpha_1, \alpha_2, \beta_1, \beta_3, \beta_4, \beta_6\}$.
  \item $G$ has no vertex-minor isomorphic to a graph in $\{C_5, \alpha_1, \beta_1\}$.
  \end{enumerate}
\end{theorem}

\begin{proof}
  By Lemma~\ref{lem:vm-rw}, $((1)\rightarrow (4))$ is clear as 
  $C_5$, $\alpha_1$ and $\beta_1$ have linear rank-width $2$.
  We can easily confirm the directions $((4)\rightarrow (3)\rightarrow (2))$; see~\cite{AdlerFP11}. 
  We add a proof for $((2)\rightarrow (1))$.

  Suppose that $G$ has linear rank-width at least $2$ and it is distance-hereditary.
  Let $D$ be the canonical split decomposition of $G$.
  By Theorem~\ref{thm:charlrw1}, $T_D$ is not a path.
  Thus there exists a bag $B$ of $D$ such that 
  $D- V(B)$ has at least three connected components $T_1$, $T_2$, $T_3$. For each $i\in \{1,2,3\}$, let $v_i:=\zeta_b(D, B,T_i)$ and $w_i:=\zeta_c(D, B,T_i)$.
  We have three cases; $B$ is a complete bag, or $B$ is a star bag with the center at one of $v_1, v_2, v_3$, or $B$ is a star bag with the center at a vertex of $V(B)\setminus \{v_1, v_2, v_3\}$.
 
  If $B$ is a complete bag, then
  $G$ has an induced subgraph isomorphic to one of $\alpha_1, \alpha_2, \alpha_3, \alpha_4$
  depending on the types of the marked edges $v_iw_i$.
  If $B$ is a star bag with the center at one of $v_1, v_2, v_3$, then $G$ has an induced subgraph isomorphic to one of 
  $\beta_1, \beta_2, \ldots, \beta_6$.
  Finally, 
  if $B$ is a star bag with the center at a vertex of $V(B)\setminus \{v_1, v_2, v_3\}$, then 
  $G$ has an induced subgraph isomorphic to one of $\gamma_1, \gamma_2, \gamma_3, \gamma_4$.
  We summarize all the cases in Table~\ref{table2}.
\end{proof}

\section{Conclusion}

In this paper we used the characterization of the linear rank-width of distance-hereditary graphs given in \cite{AdlerKK15} to
 prove that Question \ref{con:tree} is true if and only if it is true in prime graphs.
 Also, for each non-negative integer $k$, we compute a set of distance-hereditary graphs such that 
 every distance-hereditary graph of linear rank-width at least $k+1$ contains a vertex-minor isomorphic to one of the graphs in the set.

Computing an upper bound on the size of vertex-minor obstructions for graphs of bounded linear rank-width is a challenging open question. Until now only a
bound on obstructions for graphs of bounded rank-width is known \cite{Oum05}.  Secondly, resolving Question \ref{con:tree} in all graphs seems to require new techniques. 
We currently do not have any idea on how to reduce any
graph of small rank-width but large linear rank-width into a distance-hereditary graph whose decomposition tree has large path-width. One might start with graphs of rank-width $2$. 

\section*{Acknowledgment}
The authors would like to thank Isolde Adler for initial discussions on this problem.


\begin{thebibliography}{10}

\bibitem{AdlerFP11}
Isolde Adler, Arthur~M. Farley, and Andrzej Proskurowski.
\newblock Obstructions for linear rank-width at most 1.
\newblock {\em Discrete Appl. Math.}, 168:3--13, 2014.

\bibitem{AdlerK13}
Isolde Adler and Mamadou~Moustapha Kant{\'e}.
\newblock Linear rank-width and linear clique-width of trees.
\newblock {\em Theoret. Comput. Sci.}, 589:87--98, 2015.

\bibitem{AdlerKK15}
Isolde Adler, Mamadou~Moustapha Kant\'e, and O-joung Kwon.
\newblock Linear rank-width of distance-hereditary graphs {I}. {A}
  polynomial-time algorithm.
\newblock {\em Algorithmica}, 78(1):342--377, 2017.

\bibitem{BandeltM86}
Hans-J{\"u}rgen Bandelt and Henry~Martyn Mulder.
\newblock Distance-hereditary graphs.
\newblock {\em J. Comb. Theory, Ser. B}, 41(2):182--208, 1986.

\bibitem{BienstockRST91}
Daniel Bienstock, Neil Robertson, Paul~D. Seymour, and Robin Thomas.
\newblock Quickly excluding a forest.
\newblock {\em J. Comb. Theory, Ser. B}, 52(2):274--283, 1991.

\bibitem{BlumensathC10}
Achim Blumensath and Bruno Courcelle.
\newblock On the monadic second-order transduction hierarchy.
\newblock {\em Logical Methods in Computer Science}, 6(2), 2010.

\bibitem{Bouchet1987a}
Andr{\'e} Bouchet.
\newblock Isotropic systems.
\newblock {\em European J. Combin.}, 8(3):231--244, 1987.

\bibitem{Bouchet1987b}
Andr{\'e} Bouchet.
\newblock Reducing prime graphs and recognizing circle graphs.
\newblock {\em Combinatorica}, 7(3):243--254, 1987.

\bibitem{Bouchet1988}
Andr{\'e} Bouchet.
\newblock Graphic presentations of isotropic systems.
\newblock {\em J. Comb. Theory Ser. B}, 45(1):58--76, 1988.

\bibitem{Bouchet88}
Andr\'e Bouchet.
\newblock Transforming trees by successive local complementations.
\newblock {\em J. Graph Theory}, 12(2):195--207, 1988.

\bibitem{Bouchet1989a}
Andr{\'e} Bouchet.
\newblock Connectivity of isotropic systems.
\newblock In {\em Combinatorial Mathematics: Proceedings of the Third
  International Conference (New York, 1985)}, volume 555 of {\em Ann. New York
  Acad. Sci.}, pages 81--93, New York, 1989. New York Acad. Sci.

\bibitem{Bui-XuanKL13}
Binh{-}Minh Bui{-}Xuan, Mamadou~Moustapha Kant{\'{e}}, and Vincent Limouzy.
\newblock A note on graphs of linear rank-width 1.
\newblock {\em CoRR}, abs/1306.1345, 2013.

\bibitem{CorneilLB1981}
D.~G. Corneil, H.~Lerchs, and L.~Stewart Burlingham.
\newblock Complement reducible graphs.
\newblock {\em Discrete Appl. Math.}, 3(3):163--174, 1981.

\bibitem{CourcelleBanhoff08}
Bruno Courcelle.
\newblock Graph transformations expressed in logic and applications to
  structural graph theory.
\newblock Report of Banff workshop in Graph Minors (08w5079), 2008.
\newblock http://www.birs.ca/workshops/2008/08w5079/report08w5079.pdf.

\bibitem{CunninghamE80}
William~H. Cunnigham and Jack Edmonds.
\newblock A combinatorial decomposition theory.
\newblock {\em Canadian Journal of Mathematics}, 32:734--765, 1980.

\bibitem{Cunningham1982}
William~H. Cunningham.
\newblock Decomposition of directed graphs.
\newblock {\em SIAM J. Algebraic Discrete Methods}, 3(2):214--228, 1982.

\bibitem{Dahlhaus00}
Elias Dahlhaus.
\newblock Parallel algorithms for hierarchical clustering and applications to
  split decomposition and parity graph recognition.
\newblock {\em J. Algorithms}, 36(2):205--240, 2000.

\bibitem{Diestel05}
Reinhard Diestel.
\newblock {\em Graph theory}, volume 173 of {\em Graduate Texts in
  Mathematics}.
\newblock Springer, Heidelberg, fourth edition, 2010.

\bibitem{EllisST94}
Jonathan~A. Ellis, Ivan~Hal Sudborough, and Jonathan~S. Turner.
\newblock The vertex separation and search number of a graph.
\newblock {\em Inf. Comput.}, 113(1):50--79, 1994.

\bibitem{Ganian10}
Robert Ganian.
\newblock Thread graphs, linear rank-width and their algorithmic applications.
\newblock In Costas~S. Iliopoulos and William~F. Smyth, editors, {\em IWOCA},
  volume 6460 of {\em Lecture Notes in Computer Science}, pages 38--42.
  Springer, 2010.

\bibitem{GP2012}
Emeric Gioan and Christophe Paul.
\newblock Split decomposition and graph-labelled trees: characterizations and
  fully dynamic algorithms for totally decomposable graphs.
\newblock {\em Discrete Appl. Math.}, 160(6):708--733, 2012.

\bibitem{HammerM90}
Peter~L. Hammer and Fr{\'{e}}d{\'{e}}ric Maffray.
\newblock Completely separable graphs.
\newblock {\em Discrete Appl. Math.}, 27(1-2):85--99, 1990.

\bibitem{JeongKO16}
Jisu Jeong, Eun~Jung Kim, and Sang-il Oum.
\newblock Constructive algorithm for path-width of matroids.
\newblock In {\em Proceedings of the {T}wenty-{S}eventh {A}nnual {ACM}-{SIAM}
  {S}ymposium on {D}iscrete {A}lgorithms}, pages 1695--1704. ACM, New York,
  2016.

\bibitem{JKO2014}
Jisu Jeong, O-joung Kwon, and Sang-il Oum.
\newblock Excluded vertex-minors for graphs of linear rank-width at most {$k$}.
\newblock {\em European J. Combin.}, 41:242--257, 2014.

\bibitem{Kashyap2008}
Navin Kashyap.
\newblock Matroid pathwidth and code trellis complexity.
\newblock {\em SIAM J. Discrete Math.}, 22(1):256--272, 2008.

\bibitem{Thilikos2014}
Athanassios Koutsonas, Dimitrios~M. Thilikos, and Koichi Yamazaki.
\newblock Outerplanar obstructions for matroid pathwidth.
\newblock {\em Discrete Math.}, 315:95--101, 2014.

\bibitem{Oum2004a}
Sang{-}il Oum.
\newblock Rank-width and well-quasi-ordering.
\newblock {\em SIAM J. Discrete Math.}, 22(2):666--682, 2008.

\bibitem{Oum2006a}
Sang{-}il Oum.
\newblock Excluding a bipartite circle graph from line graphs.
\newblock {\em J. Graph Theory}, 60(3):183--203, 2009.

\bibitem{Oum12}
Sang{-}il Oum.
\newblock Rank-width and well-quasi-ordering of skew-symmetric or symmetric
  matrices.
\newblock {\em Linear Algebra and its Applications}, 436(7):2008 -- 2036, 2012.

\bibitem{Oum05}
Sang\mbox{-}il Oum.
\newblock Rank-width and vertex-minors.
\newblock {\em J. Comb. Theory, Ser. B}, 95(1):79--100, 2005.

\bibitem{OumS06}
Sang\mbox{-}il Oum and Paul~D. Seymour.
\newblock Approximating clique-width and branch-width.
\newblock {\em J. Comb. Theory, Ser. B}, 96(4):514--528, 2006.

\bibitem{RobertsonS83}
Neil Robertson and Paul~D. Seymour.
\newblock Graph minors. {I}. {E}xcluding a forest.
\newblock {\em J. Comb. Theory, Ser. B}, 35(1):39--61, 1983.

\bibitem{RS2004}
Neil Robertson and Paul~D. Seymour.
\newblock Graph minors. {XX}. {W}agner's conjecture.
\newblock {\em J. Comb. Theory Ser. B}, 92(2):325--357, 2004.

\bibitem{TakahashiUK94}
Atsushi Takahashi, Shuichi Ueno, and Yoji Kajitani.
\newblock Minimal acyclic forbidden minors for the family of graphs with
  bounded path-width.
\newblock {\em Discrete Math.}, 127(1-3):293--304, 1994.

\end{thebibliography}
\end{document}